\documentclass{amsart}

\usepackage{amssymb, amsmath}
\usepackage{mathrsfs}
\usepackage{amscd}
\usepackage{verbatim}
\usepackage{enumerate}
\usepackage{todonotes}

\allowdisplaybreaks

\usepackage[colorlinks,linkcolor={blue},citecolor={blue},urlcolor={red},]{hyperref}

\theoremstyle{plain}
\newtheorem{theorem}{Theorem}[section]
\theoremstyle{remark}
\newtheorem{remark}[theorem]{Remark}

\theoremstyle{plain}
\newtheorem{corollary}[theorem]{Corollary}
\newtheorem{lemma}[theorem]{Lemma}
\newtheorem{proposition}[theorem]{Proposition}
\newtheorem{definition}[theorem]{Definition}

\newtheorem{assumption}[theorem]{Assumption}

\numberwithin{equation}{section}
\def\N{{\mathbb N}}

\def\R{{\mathbb R}}
\def\C{{\mathbb C}}

\newcommand{\E}{{\mathbb E}}
\renewcommand{\P}{{\mathbb P}}

\newcommand{\B}{{\mathcal B}}

\newcommand{\F}{{\mathcal F}}
\newcommand{\calF}{{\mathcal F}}
\newcommand{\calA}{{\mathcal A}}

\newcommand{\g}{\gamma}

\renewcommand{\O}{\Omega}

\renewcommand{\Re}{\hbox{\rm Re}\,}

\newcommand{\calL}{{\mathcal L}}
\newcommand{\n}{\Vert}
\newcommand{\one}{{{\bf 1}}}
\newcommand{\embed}{\hookrightarrow}

\newcommand{\lb}{\langle}
\newcommand{\rb}{\rangle}

\newcommand{\wh}{\widehat}
\newcommand{\supp}{\text{\rm supp\,}}

\newcommand{\dom}{\mathcal{O}}
\newcommand {\ud}{d}

\def\Xint#1{\mathchoice
{\XXint\displaystyle\textstyle{#1}}%
{\XXint\textstyle\scriptstyle{#1}}%
{\XXint\scriptstyle\scriptscriptstyle{#1}}%
{\XXint\scriptscriptstyle\scriptscriptstyle{#1}}%
\!\int}
\def\XXint#1#2#3{{\setbox0=\hbox{$#1{#2#3}{\int}$ }
\vcenter{\hbox{$#2#3$ }}\kern-.6\wd0}}

\def\dashint{\Xint-}
\newcommand{\fint}{\dashint}

\newcommand{\wt}{\widetilde}
\renewcommand{\tilde}{\widetilde}
\newcommand{\DMR}{{\rm DMR}}
\newcommand{\SMR}{{\rm SMR}}
\renewcommand{\MR}{{\rm MR}}

\newcommand{\nnn}{|\!|\!|}
\newcommand {\Schw}{\mathcal{S}}

\begin{document}

\author{Pierre Portal}
\address{Australian National University\\ Mathematical Sciences Institute\\ Hannah Neumann Building 145, Canberra ACT 0200\\ Australia \\ and Universit\'e Lille 1 \\ Laboratoire Paul Painlev\'e \\
F-59655 Villeneuve d'Ascq \\ France.} \email{Pierre.Portal@anu.edu.au}

\author{Mark Veraar}
\address{Delft Institute of Applied Mathematics\\
Delft University of Technology \\ P.O. Box 5031\\ 2600 GA Delft\\The
Netherlands} \email{M.C.Veraar@tudelft.nl}

\thanks{The first author is supported by the Future Fellowship FT130100607
of the Australian Research Council.
The second author is supported by the VIDI subsidy 639.032.427 of the Netherlands Organisation for Scientific Research (NWO)}

\date\today

\title[Stochastic maximal regularity]{Stochastic maximal regularity for rough time-dependent problems}

\begin{abstract}
We unify and extend the semigroup and the PDE approaches to stochastic maximal regularity of time-dependent semilinear parabolic problems with noise given by a cylindrical Brownian motion. We treat random coefficients that are only progressively measurable in the time variable. For $2m$-th order systems with $VMO$ regularity in space, we obtain $L^{p}(L^{q})$ estimates for all $p>2$ and $q\geq 2$, leading to optimal space-time regularity results. For second order systems with continuous coefficients in space, we also include a first order linear term, under a stochastic parabolicity condition, and obtain $L^{p}(L^{p})$ estimates together with optimal space-time regularity. For linear second order equations in divergence form with random coefficients that are merely measurable in both space and time, we obtain estimates in the tent spaces $T^{p,2}_{\sigma}$ of Coifman-Meyer-Stein. This is done in the deterministic case under no extra assumption, and in the stochastic case under the assumption that the coefficients are divergence free.
\end{abstract}

\keywords{stochastic PDEs, maximal regularity, VMO coefficients, measurable coefficients, higher order equations, Sobolev spaces, $A_p$-weights}

\subjclass[2010]{Primary: 60H15, Secondary: 35B65, 42B37, 47D06.}

\maketitle

\setcounter{tocdepth}{1}
\tableofcontents

\section{Introduction}

On $X_0$ (typically  $X_0= L^r(\dom;\C^N)$ where $r\in [2, \infty)$), we consider the following stochastic evolution equation:
\begin{equation}\label{eq:SEEintro}
\left\{\begin{aligned}
dU(t)  +  A(t) U(t)\ud t & =  F(t,U(t)) \ud t  +  \big(B(t) U(t)  +G(t,U(t))\big) \ud W_H(t),\\
 U(0) & = u_0,
\end{aligned}
\right.
\end{equation}
where $H$ is a Hilbert space, $W_{H}$ a cylindrical Brownian motion, $A:\R_+\times\O\to \calL(X_1, X_0)$ (for some Banach space $X_{1}$ such that $X_1\hookrightarrow X_0$, typically a Sobolev space) and $B:\R_+\times \O\to \calL(X_{1}, \gamma(H,X_{\frac12}))$ are progressively measurable (and satisfy a suitable stochastic parabolic estimate), the functions $F$ and $G$ are suitable nonlinearities, and the initial value $u_0:\O\to X_{0}$ is $\F_0$-measurable (see Section \ref{sec:MR} for precise definitions). We are interested in {\em maximal $L^p$-regularity} results for \eqref{eq:SEEintro}. This means that we want to investigate well-posedness together with sharp $L^p$-regularity estimates given the data $F,G$ and $u_0$.

Knowing these sharp regularity results for equations such as \eqref{eq:SEEintro}, gives a priori estimates to nonlinear equations involving suitable nonlinearities $F(t,U(t))\ud t$ and $G(t,U(t))\ud W_H(t)$. Well-posedness of such non-linear equations follows easily from these a priori estimates (see e.g.\ the proofs in \cite{NVWsiam}).

\subsection{Deterministic maximal regularity}
In deterministic parabolic PDE, maximal regularity is routinely used without identifying it as a specific property. It is traditionally established by showing that the kernel of the semigroup is a standard Calder\'on-Zygmund kernel as a function of space and time (see e.g.\ \cite{LSU} for a general theory and \cite{kt} for a quintessential example).
As a property of abstract evolution equations, maximal regularity also has a long history. A turning point, that can be seen as the starting point for the methods used in this paper, was reached in \cite{We}. He obtained a characterisation of maximal regularity in the time-independent case:
the deterministic problem
$$
U'(t) +AU(t)= f(t).
$$
Under the assumption that the underlying space $X_0$ is UMD, he showed that $A$ has maximal $L^p$-regularity if and only if $A$ is $R$-sectorial.

In the time-dependent case, maximal regularity is far less understood.
For abstract evolution equations, it has been established under regularity assumptions in time: continuity when $D(A(t))$ is constant (see \cite{PrSch} and the references therein), and H\"older regularity when $D(A(t))$ varies (see \cite{ps}). For concrete PDE with very general boundary conditions, it has been established under continuity assumptions on the coefficients in \cite{DHP} and this was extended to equations with VMO coefficients in time and space in \cite{DongGal}. For equations with Dirichlet boundary conditions one can obtain maximal regularity when the coefficients satisfy a VMO condition in space and are measurable in time, see \cite{DK11, DKnew, Kry2008} and references therein.

In Section \ref{sec:2morder} we apply the results of \cite{DKnew} to obtain an $L^p(L^q)$-theory for higher order systems. In Section \ref{sec:secondorder} we consider second order systems and we use a more classical technique: we obtain an $L^p(L^q)$ in the space independent case first, then for $p=q$ we use standard localization arguments to reach the space dependent case under minimal regularity assumptions in the spatial variable.

Treating fully rough coefficients (merely bounded and measurable in both space and time) seems to be much more difficult, despite the fact that it was already understood in $L^{2}(\R_{+};L^{2}(\R^{d}))$ in the 1950's. The form method developed by J.L. Lions and his school allows one, in that case, to prove the following variant of maximal regularity:
\begin{equation*}
\|u\|_{W^{1,2}(I;W^{-1,2}(\R^{d}))} + \|u\|_{L^{2}(I;W^{1,2}(\R^{d}))}\leq C\|f\|_{L^{2}(I;W^{-1,2}(\R^{d}))}.
\end{equation*}
In Section \ref{sec:tent}, we initiate the development of Lions's maximal regularity theory in appropriate analogues of Lions's energy space $L^{2}(I;W^{1,2}(\R^{d}))$. These spaces are the tent spaces $T^{p,2}$ introduced in \cite{CMS}, and extensively used in harmonic analysis of PDE with rough coefficients (see e.g. \cite{as, aa, hkmp} and the references therein). It was discovered in \cite{AMP15} that Lions's well posedness theory for time-dependent divergence form parabolic problems with $L^{\infty}(\R_{+}\times \R^{d})$ coefficients can be extended to tent spaces. Here we start extending the corresponding maximal regularity, both in the deterministic and in the stochastic setting.\\

For deterministic equations maximal $L^p$-regularity can be used to obtain a local existence theory for quasilinear PDEs of parabolic type (see \cite{CleLi, Pruss02, PruSimbook}). Moreover, it can sometimes be used to derive global existence for semilinear equations (see \cite{Pierre, Pruss02}). In \cite{KPW} maximal $L^p$-regularity was used to study long time behavior of solutions to quasilinear equations. In \cite{PruSimWil} it was used to study critical spaces of initial values for which the quasilinear equation is well-posed.

At the moment it remains unclear which of the mentioned theories have a suitable version for stochastic evolution equations. In this paper we develop a maximal $L^p$-regularity theory for \eqref{eq:SEEintro} which extends several of existing known theories. In future works we plan to study consequence for concrete nonlinear SPDEs. In the next subsection we explain some of the known results, and then compare them to what is proved in the current paper.

\subsection{SPDEs of second order}\label{subs:Krylovapproach}
For {\em second order elliptic operators} $A$ on $\R^d$ in non-divergence form, this theory was first developed by Krylov in a series of papers \cite{Kry94a,Kry,Kry96,Kry00} and was surveyed in \cite{Kry06,KryOverview}. These works have been very influential. In particular they have led to e.g.\ \cite{Du2018,Kimildoo15,Kim04a,Kim04b,Kim05,Kim09,KimKimquasi2,KimKimKim13,KryLot} where also the case of smooth domains has been considered, and later to e.g.\  \cite{CioicaKimLeeLindner13,CioicaKimLeeLindner18,CioicaKimLee18,Kim14nonsmooth,Lindner14} where the case of non-smooth domains is investigated. In the above mentioned results one uses $L^p$-integrability in space, time and $\Omega$. In \cite{Kim09,Kry00} $p\neq q$ is allowed but only if $q\leq p$.

The above mentioned papers mostly deal with second order operators of scalar equations. In the deterministic setting higher order systems are considered as well (see e.g.\ \cite{DK11, DKnew, GV, GV2, Kry2008}). In the stochastic case some $L^p$-theory for second order systems has been developed in \cite{KimLeesystems,MiRo01} and an $L^p(\Omega;L^2)$-theory in \cite{DuLiuZhang}, but in the last mentioned paper the main contribution is a $C^{\alpha}$-theory.

\subsection{The role of the $H^\infty$-calculus assumptions}\label{ssubs:Hinftyapproach}
In \cite{NVW12a,NVWsiam,NVW15} together with van Neerven and Weis, the second author has found another approach to maximal $L^p$-regularity of SPDEs, based on McIntosh's $H^\infty$-calculus and square function estimates (see e.g.\ \cite{CDMY,Haase:2,KW,KWcalc,McI}). This allows one to obtain maximal $L^p$-regularity for \eqref{eq:SEEintro} for any sectorial operator $A$ on $L^q$-spaces ($q\geq2$) with a bounded $H^\infty$-calculus of angle $<\pi/2$. There is a vast literature with examples of operators with a bounded $H^\infty$-calculus (see \cite[Section 10.8]{HNVW2} for a more complete list):
\begin{itemize}
\item \cite{DDHPV}: $2m$-th order elliptic systems with general boundary conditions and smooth coefficients
\item \cite{DuongYan}: second order elliptic equations with VMO coefficients.
\item \cite{E18} second order elliptic systems in divergence form on bounded Lipschitz domains,
with $L^{\infty}$ coefficients and mixed boundary conditions.
\item \cite{KuWe17} Stokes operator on a Lipschitz domain
\item \cite{LinVerLp}: Dirichlet Laplace operator on $C^2$-domains with weights
\item \cite{MM16} Hodge Laplacian and Stokes operator with Hodge boundary conditions on very weakly Lipschitz domains
\end{itemize}
One advantage of the above approach is that it leads to an $L^p(\Omega\times(0,T);L^q)$-theory for all $p\in (2, \infty)$ and $q\in [2, \infty)$ (where in case $q=2$, the case $p=2$ is included), and gives optimal space-time regularity results such as $U\in L^p(\Omega;H^{\theta,p}(0,T;X_{1-\theta}))$ or even $U\in L^p(\Omega;C([0,T];X_{1-\frac1p,p}))$, where we used complex and real interpolation space respectively. Such results seem unavailable in the approach of Subsection \ref{subs:Krylovapproach}.

\subsection{New results}
Until now the approach based on functional calculus techniques was limited to equations where $A$ was independent of time and $\Omega$ (or continuous in time see \cite{NVWsiam}). We will give a simple method to also treat the case where the coefficients of the differential operator $A$ only depend on time and $\Omega$ in a progressive measurable way. The method is inspired by \cite[Lemma 5.1]{Kry09VMO} and \cite{KimLeesystems} where it is used to reduce to the case of second order equations with constant coefficients. \\

Our paper extends and unifies the theories in \cite{Kry} and \cite{NVW12a} in several ways. Moreover, we introduce weights in time in order to be able to treat rough initial values. In the deterministic setting weight in time have been used for this purpose in \cite{PRSimweight}. In the stochastic case some result in this direction have been presented in \cite{AndJentzenKu}, but not in a maximal regularity setting. Furthermore, we initiate a Lions's type stochastic maximal regularity theory outside of Hilbert spaces, based on the $L^2$ theory (see \cite{L},  \cite{Par2}, \cite{KR79}, \cite{LiuRock}), \cite{Rozov}, \cite{AMP15}, and \cite{anp}. Our main abstract results can be found in Theorems \ref{thm:timedepSMR} and \ref{thm:nonlinearperturbation} below. Our result in the Lions's setting is Theorem \ref{thm:tent}.\\

Additionally we are able to give an abstract formulation of the stochastic parabolicity condition for $A$ and $B$ (see Section \ref{subs:B=0}). It coincides with the classical one if $A$ is a scalar second order operator on $\R^d$ and $B$ consists of first order operators.

In the applications of our abstract results we will only consider equations on the full space $\R^d$, but in principle other situations can be considered as well. However, in order to include an operator $B$ which satisfies an optimal abstract stochastic parabolicity condition, we require certain special group generation structure.

The concrete SPDEs we consider are
\begin{itemize}
\item $2m$-th order elliptic systems in non-divergence form with coefficients which are only progressively measurable (see Theorem \ref{thm:2msystemxind}). The main novelties are that, in space, the coefficients are assumed to be VMO, and we are able to give an $L^p(L^q)$-theory for all $p\in (2, \infty)$ and $q\in [2, \infty)$ ($p=q=2$ is allowed as well).
\item Second order elliptic systems in non-divergence form with coefficients which are only progressively measurable, with a diffusion coefficient that satisfies an optimal stochastic parabolicity condition (see Theorem \ref{thm:secondorderxind}). When the coefficients are independent of space we give an $L^p(L^q)$-theory. Moreover, we give an $L^p(L^p)$-theory if the coefficients are continuous in space.
\item
Second order divergence form equations with coefficients which are only progressively measurable in both the time and the space variables, but satisfy the structural condition of being divergence free. We treat this problem in suitable tent spaces, and in the model case where $B=0$, $u_{0}=0$.
\end{itemize}
A major advantage of our approach to the $L^p(L^q)$-theory, is that we can obtain the same space-time regularity results as in Section \ref{ssubs:Hinftyapproach}. This seems completely new in the case of measurable dependence on $(t,\omega)$.
Our approach to stochastic maximal regularity in the Lions sense is nowhere as developed, but gives, to the best of our knowledge, the first results (outside of Hilbert spaces) where no regularity in either space and time is assumed.

\subsection{Other forms of maximal regularity}
To end this introduction let us mention several other type of maximal $L^p$-regularity results. In \cite{Brz1,DPL} maximal $L^p$-regularity for any analytic semigroup was established in the real interpolation scale. In \cite{vNVWgamma} maximal regularity was obtained using $\gamma$-spaces. In Banach function spaces variations of the latter have been obtained in \cite{Antoni}.

\subsubsection*{Notation}

We write $A \lesssim_p B$ whenever $A \leq C_p B$ where $C_p$ is a constant which depends on the parameter $p$. Similarly, we write $A\eqsim_p B$ if $A\lesssim_p B$ and $B\lesssim_p A$. Moreover, $C$ is a constant which can vary from line to line.

\subsubsection*{Acknowledgements}

This work was started during a visit of Veraar at the Australian National University in 2016, and concluded during a visit of Portal at the TU Delft in 2018. The authors would like to express their gratitude to these institutions for providing an excellent environment for their collaboration. We would like to thank Antonio Agresti for carefully reading the manuscript.

\section{Preliminaries}

\subsection{Measurability\label{subsec:meas}}

Let $(S,\Sigma, \mu)$ be a measure space. A function $f:S\to X$ is called {\em strongly measurable} if it can be approximated by $\mu$-simple functions a.e.
An operator valued function $f:S\to \calL(X,Y)$ is called $X$-strongly measurable if for every $x\in X$, $s\mapsto f(s) x$ is strongly measurable.

Let $(\Omega, \P, \calA)$ be a probability space with filtration $(\calF_t)_{t\geq 0}$. A process $\phi:\R_+\times\O\to X$ is called {\em progressively measurable} if for every fixed $T\geq 0$, $\phi$ restricted to $[0,T]\times\O$ is strongly $\B([0,T])\times \F_T$-measurable

An operator valued process $\phi:\R_+\times\O\to \calL(X,Y)$ will be called {\em $X$-strongly progressively measurable} if for every $x\in X$, $\phi x$ is progressively measurable.

Let $\bigtriangleup:=\{(s,t): 0\leq s\leq t<\infty\}$ and $\bigtriangleup_T = \bigtriangleup\cap [0,T]^2$. Let $\B_T$ denotes the Borel $\sigma$-algebra on $\bigtriangleup_T$.
A two-parameter process $\phi:\bigtriangleup\times\O\to X$ will be called {\em progressively measurable} if for every fixed $T\geq 0$,
$\phi$ restricted to $\bigtriangleup_T\times\O$ is strongly $\B_T\times \F_T$-measurable.

\subsection{Functional calculus}
For $\sigma\in (0,\pi)$ let $\Sigma_{\sigma} = \{z\in \C: |\arg(z)|<\sigma\}$.
A closed and densely defined operator $(A,D(A))$ on a Banach space $X$ is called {\em sectorial of type $(M,\sigma)\in \R_+\times (0,\pi)$} if $A$  is injective, has dense range, $\sigma(A)\subseteq \overline{\Sigma_{\sigma}}$ and
\begin{align*}
\|\lambda R(\lambda,A)\|\leq M, \qquad  \ \ \lambda\in \C\setminus \Sigma_{\sigma}.
\end{align*}
A closed and densely defined operator $(A,D(A))$ on a Banach space $X$ is called {\em sectorial of type $(M,w,\sigma)\in \R_+\times\R\times (0,\pi)$} if $A+w$ is sectorial of type $(M, \sigma)$.

Let $H^\infty(\Sigma_{\varphi})$ denote the space of all bounded holomorphic functions $f:\Sigma_{\varphi}\to \C$ and let $\|f\|_{H^\infty(\Sigma_{\varphi})} = \sup_{z\in \Sigma_{\varphi}} |f(z)|$. Let $H^\infty_0(\Sigma_{\varphi})\subseteq H^\infty(\Sigma_{\varphi})$  be the set of all $f$ for which there exists an $\varepsilon>0$ and $C>0$ such that $|f(z)|\leq C \frac{|z|^{\varepsilon}}{1+|z|^{2\varepsilon}}$.

For an operator $A$ which is sectorial of type $(M,\sigma)$, $\sigma<\nu<\varphi$, and $f\in H^\infty_0(\Sigma_{\varphi})$ define
\[f(A) = \frac{1}{2\pi i} \int_{\partial \Sigma_{\nu}} f(\lambda) R(\lambda,A) \ud \lambda,\]
where the orientation is such that $\sigma(A)$ is on the right side of the integration path.
The operator $A$ is said to have a {\em bounded $H^\infty$-calculus} of angle $\varphi$ if there exists a constant $C$ such that for all $f\in H^\infty_0(\Sigma_{\varphi})$
\[\|f(A)\|\leq C\|f\|_{H^\infty(\Sigma_{\varphi})}.\]
For details on the $H^\infty$-functional calculus we refer the reader to \cite{Haase:2} and \cite{HNVW2}. A list of examples has been given in the introduction.

For an interpolation couple $(X_0,X_1)$ let
\[X_{\theta} = [X_0, X_1]_{\theta}, \ \  \text{and} \ \ X_{\theta,p} = [X_0, X_1]_{\theta,p}\]
denote the complex and real interpolation spaces at $\theta\in (0,1)$ and $p\in [1, \infty]$, respectively.

\subsection{Function spaces}

Let $S\subseteq \R^d$ be open. For a weight function $w:\R^d\to (0,\infty)$ which is integrable on bounded subset of $\R^d$, $p \in [1,\infty)$, and $X$ a Banach space, we work with the Bochner spaces $L^{p}(S,w;X)$ with norm defined by
$$
\|u\|_{L^{p}(S,w;X)} ^{p} = \int \limits _{S} \|u(t)\|_{X} ^{p} w(t)\ud t,
$$
We also use the corresponding Sobolev spaces defined by
$$\|u\|_{W^{1,p}(S,w;X)} ^{p}= \|u\|_{L^{p}(S,w;X)} ^{p} + \|u'\|_{L^{p}(S,w;X)} ^{p}.
$$

If $q<p$,  and $w_{\alpha}(x) = |x|^{\alpha}$ with $\alpha/d <\frac{p}{q}-1$, note that, by H\"older inequality $L^{p}(S,w_{\alpha};X)\hookrightarrow L^q(S;X)$.

In several cases the class of weight we will consider is the class of $A_p$-weights $w:\R^d\to (0,\infty)$. Recall that $w\in A_p$ if and only if the Hardy--Littlewood maximal function is bounded on $L^p(\R^d,w)$.

For $p\in (1, \infty)$ and an $A_p$-weight $w$ let the Bessel potential spaces $H^{s,p}(\R^d,w;X)$ be defined as the space of all $f\in \Schw'(\R^d;X):=\calL(\Schw(\R^d),X)$ for which $\F^{-1} [(1+|\cdot|^2)^{s/2} \wh{f} \in L^p(\R^d,w;X)$. Here $\F$ denotes the Fourier transform. Then $H^{s,p}(\R^d,w;X)$ is a Banach space when equipped with the norm
\[\|f\|_{H^{s,p}(\R^d,w;X)} = \|\F^{-1} [(1+|\cdot|^2)^{s/2} \wh{f} ]\|_{L^p(\R^d,w;X)}.\]
The following is a well known consequence of Fourier multiplier theory.
\begin{lemma}\label{lem:equivalentnormsH}
Let $X$ be a UMD Banach space, $p\in (1, \infty)$, $s\in \R$, $r>0$ and $k\in \N$.
Then the following give equivalent norms on $H^{s,p}(\R^d;X)$:
\begin{align*}
& \|(-\Delta)^{r/2} u\|_{H^{s-r,p}(\R^d;X)} + \|u\|_{H^{s-r,p}(\R^d;X)},
\\ &\sum_{|\alpha| = k} \|\partial^{\alpha} u\|_{H^{s-k,p}(\R^d;X)} + \|u\|_{H^{s-k,p}(\R^d;X)}.
\end{align*}
\end{lemma}

The spaces $H^{s,p}$ will also be needed on bounded open intervals $I$. For a $I\subseteq \R$, $p\in (1, \infty)$, $w\in A_p$, $s\in \R$ the space $H^{s,p}(I,w;X)$ is defined as all restriction $f|_{I}$ where $f\in H^{s,p}(I,w;X)$. This is a Banach space when equipped with the norm
\[\|f\|_{H^{s,p}(I,w;X)} = \inf\{\|g\|_{H^{s,p}(\R,w;X)}: g|_{I} = f\}.\]

Either by repeating the proof of Lemma \ref{lem:equivalentnormsH} or by reducing to it by applying a bounded extension operator from $H^{\theta,p}(I,w;Y)\to H^{\theta,p}(\R,w;Y)$ and Fubini, we obtain the following norm equivalence.
\begin{lemma}\label{lem:equivalentnormsH2}
Let $X$ be a UMD space, $p\in (1, \infty)$, $s\in \R$, $r>0$, $k\in \N$, and let $I\subseteq \R$ be an open interval. Let $\theta\in (0,1)$ and $w\in A_p$. Then the following two norms give equivalent norms on $H^{\theta, p}(I;H^{s,p}(\R^d;X))$:
\begin{align*}
& \|(-\Delta)^{r/2} u\|_{H^{\theta,p}(I,w;H^{s-r,p}(\R^d;X))} + \|u\|_{H^{\theta,p}(I;H^{s-r,p}(\R^d;X))},
\\ &\sum_{|\beta| = k} \|\partial^{\beta} u\|_{H^{\theta, p}(I;H^{s-k,p}(\R^d;X))} + \|u\|_{H^{\theta,p}(I;H^{s-k,p}(\R^d;X))}.
\end{align*}
\end{lemma}

The next result follows from \cite[Proposition 7.4]{MeyVer}.
\begin{proposition}\label{prop:Sobolevemb}
Let $p\in (1, \infty)$, $\alpha\in [0,p-1)$, $T\in (0,\infty]$ and set $I=(0,T)$.
For all $f\in H^{\theta,p}(I,t^{\alpha};X)$ we have
\begin{align*}
\|f\|_{C^{\theta-\frac{1+\alpha}{p}}(\overline{I};X)}& \leq C\|f\|_{H^{\theta,p}(I,w_{\alpha};X)} \ \text{if} \  \ \theta>\frac{1+\alpha}{p},
\\  \|f\|_{C^{\theta-\frac1p}([\varepsilon,T];X)}& \leq C_{\varepsilon}\|f\|_{H^{\theta,p}(I,w_{\alpha};X)} \ \ \text{if} \ \ \theta>\frac{1}{p},  \ \ \varepsilon\in (0,T].
\end{align*}
\end{proposition}

\begin{proposition}\label{prop:funcspacesUMD}
Let $X_0, X_1$ be UMD spaces and assume $(X_0, X_1)$ is an interpolation couple. Let $p\in (1, \infty)$, $w\in A_p$, and let $I\subseteq \R$ be an open interval. If $s_0<s_1$, $\theta\in (0,1)$ and $s = (1-\theta)s_0 + \theta s_1$, then the following assertions hold:
\begin{enumerate}[(1)]
\item\label{it1:funcspacesUMD} $W^{1,p}(I,w;X_0) = H^{1,p}(I,w;X_0)$.
\item\label{it2:funcspacesUMD} $[H^{s_0, p}(I,w;X_0), H^{s_1, p}(I,w;X_1)]_{\theta} = H^{s,p}(I,w;[X_0, X_1]_{\theta})$.
\end{enumerate}
In particular, there exists a constant $C$ such that for any $f\in H^{s_1, p}(I,w;X_0\cap X_1)$,
\[\|f\|_{H^{s,p}(I,w;[X_0, X_1]_{\theta})}\leq C\|f\|_{H^{s_0,p}(I,w;X_0)}^{1-\theta} |f\|_{H^{s_1,p}(I,w;X_1)}^{\theta}.\]
\end{proposition}
\begin{proof}
\eqref{it1:funcspacesUMD}: This can be proved as in \cite[Proposition 5.5]{LMV} by using a suitable extension operator and a suitable extension of $w|_{I}$ to a weight on $\R$.

\eqref{it2:funcspacesUMD}: For $I = \R$, this follows from \cite[Theorem 3.18]{LinVerLp}. The general case follows from an extension argument as in \cite[Proposition 5.6]{LMV}.
\end{proof}

The following result follows from \cite[Theorem 1.1]{MeyVerTr} and standard arguments (see \cite{AgrVer} for details). Knowing the optimal trace space is essential in the proof of Theorem \ref{thm:nonlinearperturbation}.
\begin{proposition}[Trace embedding]\label{prop:traceembedding}
Let $X_0$ be UMD Banach spaces and $A$ a sectorial operator on $X_0$ and $0\in \rho(A)$ with $D(A) = X_1$. Let $p\in (1, \infty)$, $\alpha\in [0, p-1)$, $\beta\in (0,1)$ and $T\in (0,\infty]$. Set $w_{\alpha}(t) = t^{\alpha}$, $I = (0,T)$. Let
\[X_{\theta} = [X_0, X_{1}]_{\theta},   \ \ X_{\theta,p} =  (X_0, X_1)_{\theta,p}\]
denote the complex and real interpolation spaces for $\theta\in (0,1)$.
Then
\begin{align*}
L^p(I,w_{\alpha};X_{1})\cap W^{1,p}(I,w_{\alpha};X_0) & \hookrightarrow BUC(\overline{I}; X_{1-\frac{1+\alpha}{p},p}),
\\ L^p(I,w_{\alpha};X_{\beta})\cap H^{\beta,p}(I,w_{\alpha};X_0)& \hookrightarrow BUC(\overline{I};X_{\beta-\frac{1+\alpha}{p},p}).
\end{align*}
\end{proposition}

In the one-dimensional case we will also need the much simpler fractional Sobolev--Sobolewski spaces on $I=(0,T)$ with $T\in (0,\infty]$. For $\beta\in (0,1)$, $p\in (1, \infty)$ and a weight $w\in A_p$ we define the fractional Sobolev--Sobolewski space $W^{\beta,p}(I,w;X)$ as the space of all functions $\phi\in L^p(I,w;X)$ for which
\begin{align}\label{eq:Wthetanorm}
[\phi]_{W^{\beta,p}(I,w;X)}^p = \int_0^T \int_0^{T-h} \|\phi(s+h)-\phi(s)\|^p w(s) h^{-\beta p -1} \ud s \ud h<\infty.
\end{align}
This space is a Banach space when equipped with the norm $\|\phi\|_{W^{\beta,p}(I,w;X)} = [\phi]_{W^{\beta,p}(I,w;X)} + \|\phi\|_{L^p(I,w;X)}$. In the case $w_{\alpha}(t) = t^{\alpha}$ with $\alpha\in [0,p-1)$ it is well-known that (see \cite{Grisvard} and \cite[Proposition 1.1.13]{meyries2010maximal})
\begin{equation}\label{eq:realinterpol}
W^{\beta,p}(I,w_{\alpha};X) = (L^p(I,w_{\alpha};X),W^{1,p}(I,w_{\alpha};X))_{\beta, p}.
\end{equation}
For general $A_p$-weights such a characterization seems only possible if \eqref{eq:Wthetanorm} is replaced by a more complicated expression (see \cite[Proposition 2.3 with $p=q$]{MeyVerTr} for the case $I = \R$).

Note that, by \eqref{eq:realinterpol}, Proposition \ref{prop:funcspacesUMD} and general properties of real and complex interpolation \cite[Theorems 1.3.3(e) and 1.10.3]{Tr1}, we have
\begin{equation}\label{eq:realcomplexconnect}
W^{\beta,p}(I,w_{\alpha};X)\hookrightarrow H^{\theta,p}(I,w_{\alpha};X)
\end{equation}
for any UMD space $X$, $p\in (1, \infty)$ and $0<\theta<\beta<1$.

\subsection{Stochastic integration}

Let $L^p_{\F}(\Omega;L^q(I;X))$ denote the space of progressively measurable processes in $L^p(\Omega; L^q(I;X))$.

The It\^o integral of an {\em $\F$-adapted finite rank step process in $\gamma(H,X)$}, with respect to an $\F$-cylindrical Brownian motion $W_H$, is defined by
$$  \int_{\R_+} \sum \limits _{k=1} ^{N} \sum_{j=1}^M \one_{(t_{k},t_{k+1}] \times F_{k}}\otimes (h_{j}\otimes x_{k}) \,dW_H:=
\sum \limits _{k=1} ^{N} \sum_{j=1}^M \one_{F_{k}} [W_{H}(t_{k+1})h_{j}-W_{H}(t_{k})h_{j}]\otimes x_{k}, $$
for $N \in \N$, $0\leq t_{1}<t_{2}<...<t_{N+1}$, and for all $k=1,...,N$,
$F_{k}\in \F_{t_{k}}$, $h_{k}\in H$, $x_{k}\in X$. The following version of It\^o's isomorphism holds for such processes (see \cite{NVW1}):
\begin{theorem}\label{thm:UMD}
Let $X$ be a UMD Banach space and let $G$ be an $\F$-adapted
finite rank step process in $\gamma(H,X)$. For all $p\in (1,\infty)$
one has the two-sided estimate
\begin{align}\label{eq:twosided}
\E\sup_{t\geq 0}\Big\|\int_{0}^t G(s) \, dW_H(s)\Big\|^p \eqsim_p
\E\|G\|_{\gamma(L^2(\R_+;H),X))}^p,
\end{align}
with implicit constants depending only on $p$ and (the UMD constant of) $X$.
\end{theorem}

The class of UMD Banach spaces includes all Hilbert spaces, and all $L^q(\dom;G)$ spaces for $q \in (1,\infty)$, and $G$ another UMD space.
It is stable under isomorphism of Banach spaces, and included in the class of reflexive Banach spaces. Closed subspaces, quotients, and duals of UMD spaces are UMD. For more information on UMD spaces see \cite{HNVW1} or \cite{Burk01}.\\

Theorem \ref{thm:UMD} allows one to extend the stochastic integral, by density, to the closed linear span in $L^p(\Omega;\gamma(L^2(\R_+;H),X))$ of all $\F$-adapted finite rank step processes in $\gamma(H,X)$) (see \cite{NVW1}). We denote this closed linear span by $L^p _{\mathcal{F}} (\Omega;\gamma(L^2(\R_+;H),X))$. Moreover, this set coincides with the progressively measurable processes in $L^p(\Omega;\gamma(L^2(\R_+;H),X))$. \\

If the UMD Banach space $X$ has type $2$ (and thus martingale type 2), then one has a continuous embedding
$
L^2(\R_+;\g(H,X))\embed \gamma(L^2(\R_+;H),X)
$
(see \cite{NW1, RS}). See \cite{HNVW2} or \cite{DJT, Pi75, Pi2} for a presentation of the notions of type and martingale type.

In such a Banach space, \eqref{eq:twosided} implies that
\begin{align}\label{eq:type2est}
\E\sup_{t\geq 0}\Big\|\int_{0}^t G(s) \, dW_H(s)\Big\|^p \leq C^p
\E\|G\|_{L^2(\R_+;\g(H,X))}^p,
\end{align}
where $C$ depends on $X$ and $p$. The stochastic integral thus uniquely extends to
$L_\F^p(\O;L^2(\R_+;\g(H,X)))$ (as it does in \cite{Brz2, Nh}).

Note, however, that the sharp version of It\^o's isomorphism given in Theorem \ref{thm:UMD} is critical to prove stochastic maximal regularity, even in time-independent situations. The weaker estimate \eqref{eq:type2est} (where the right hand side would typically be
$L^{2}(\R_{+};L^{p}(\R^{d}))$ instead of $L^{p}(\R^{d};L^{2}(\R_{+}))$) does not suffice for this purpose (see \cite{NVW12a}).\\

We end this subsection with a simple lemma which is applied several times. It will be stated for weights in the so-called $A_q$ class in dimension one. In the unweighted case the lemma is simple and well-known. Note that $w(t) = |t|^{\alpha}$ is in $A_q$ if and only if $\alpha\in (-1, q-1)$.
\begin{lemma}\label{lem:simpleestt}
Assume $X$ is a UMD space with type $2$. Let $p\in [2, \infty)$, and $w\in A_{\frac{p}{2}}$ (if $p=2$, then we take $w=1$),  $\theta\in (0,\frac12)$, and $T\in (0,\infty)$ and set $I = (0,T)$. Assume
\[U(t)  = u_0 + \int_0^t f(s) \ud s + \int_0^t g(s) \ud W_H(s), \ \ \ t\in [0,T].\]
where $f\in L^p_{\F}(\Omega;L^p(I,w;X))$ and $g\in L^p_{\F}(\Omega;L^p(I,w;\gamma(H,X)))$. Then $U\in L^p(\Omega;C(\overline{I};X))$ and there exists a constant $C = C(p,w,T,X,\theta)$ which is increasing in $T$ and such that
\begin{align*}
\|U&\|_{L^p(\Omega;C(\overline{I};X))} + \|U\|_{L^p(\Omega;W^{\theta,p}(I, w;X))}\\ & \leq \|u_0\|_{X} + C\|f\|_{L^p(\Omega;L^p(I,w;X))} + C\|g\|_{L^p(\Omega;L^p(I,w;\gamma(H,X)))}.
\end{align*}
\end{lemma}
\begin{proof}
The definition of a strong solution and the properties of the integrals immediately give the existence of a continuous modification, and by \eqref{eq:type2est}, we find
\begin{align*}
\|&U\|_{L^p(\Omega;C(\overline{I};X))} \\ & \lesssim_{X,p} \|u_0\|_{L^p(\Omega;X)} + \|f\|_{L^p(\Omega;L^1(I;X))} +C\|g\|_{L^p(\Omega;L^2(I;\gamma(H,X)))}
\\ & \lesssim_{p,w,T} \|u_0\|_{X} + \|f\|_{L^p(\Omega;L^p(I,w;X))} + \|g\|_{L^p(\Omega;L^p(I,w;\gamma(H,X)))},
\end{align*}
where in the last step we applied H\"older's inequality.

To prove the estimate concerning the fractional regularity note that
\[\|u_0\|_{W^{\theta,p}(I,w;X)}\leq \|u_0\|_{L^{p}(I,w;X)} \leq C \|u_0\|_{X}\]
and
\[\Big\|\int_0^{\cdot} f(s) \ud s\Big\|_{W^{\theta,p}(I,w;X)}\leq
\Big\|\int_0^{\cdot} f(s) \ud s\Big\|_{W^{1,p}(I,w;X)} \leq C\|f\|_{L^p(I,w;X)}\]
which gives the required estimates after taking $L^p(\Omega)$-norms.

Let $I(g) = \int_0^\cdot g d W_H$. By \eqref{eq:type2est} and H\"older inequality the stochastic integral can be estimated, for $t \in I$, by
\[\|I(g)(t)\|_{L^p(\Omega;X)}\leq C\|g\|_{L^p(\Omega;L^p((0,T),w;\gamma(H,X)))}.\]
Taking $L^p((0,T),w)$ norms, part of the required estimate follows. For the difference norm part, first consider $p\in (2, \infty)$. Then, for $s \in I$, and $M$ denoting the Hardy-Littlewood maximal function,
\begin{align*}
\|I(g)(s+h) - I(g)(s)\|_{L^p(\Omega;X)} & \leq C \|g\|_{L^p(\Omega;L^2((s,s+h);\gamma(H,X)))}
\\ & \leq C h^{1/2} \|(M \|g\|^2_{\gamma(H,X)})^{1/2}\|_{L^p(\Omega)}.
\end{align*}
Therefore, from \eqref{eq:Wthetanorm} we obtain
\begin{align*}
[I(g)]_{L^p(\Omega;W^{\theta,p}(I,w;X))}^p &= \int_0^T \int_{0}^{T-h} \|I(g)(s+h) - I(g)(s)\|_{L^p(\Omega;X)}^p w(s) h^{-\theta p -1} \ud s \ud h
\\ & \leq C  \int_0^T \|(M \|g\|^2_{\gamma(H,X)})^{1/2}\|_{L^p(\Omega;L^p(I,w))}^p h^{(\frac12-\theta) p -1} \ud h
\\ & \leq C \|g\|_{L^p(\Omega;L^p(I,w;\gamma(H,X)))}^p,
\end{align*}
where we used $\theta\in (0,\frac12)$ and applied the boundedness of the maximal function in $L^{p/2}(\R,w)$ (see \cite[Theorem 9.1.9]{GrafakosM}).

If $p=2$, then $w=1$ and we can write (using again that $\theta<\frac{1}{2}$)
\begin{align*}
\int_0^T &\int_0^{T-h}\|I(g)(s+h) - I(g)(s)\|_{L^2(\Omega;X))}^2 \ud s \ud h \\ & \leq C \int_{0}^T \int_{0}^{T-h} \int_s^{s+h} \|g(\sigma)\|_{L^2(\Omega;\gamma(H,X)))}^2 h^{-2\theta  -1} \ud \sigma \ud s \ud h
\\ & \leq C \int_{0}^T \int_{\sigma}^{T} \int_{\sigma-s}^{T-s}  h^{-2\theta  -1} \ud h \ud s  \, \|g(\sigma)\|_{L^2(\Omega;\gamma(H,X)))}^2\ud \sigma
\\ & \leq C \int_{0}^T \int_{0}^{\sigma} (\sigma-s)^{-2\theta} \ud s  \, \|g(\sigma)\|_{L^2(\Omega;\gamma(H,X)))}^2\ud \sigma
\\ & \leq C\|g\|_{L^2(\Omega\times I;\gamma(H,X)))}^2.
\end{align*}
\end{proof}

\begin{remark}
Fractional regularity of stochastic integrals in the vector-valued setting is considered in many previous papers (see \cite{Brz2,OndrVer,ProVer15} and references therein). In particular, the unweighted case of Lemma \ref{lem:simpleestt} can be found in \cite[Corollary 4.9]{ProVer15} where it is a consequence of a regularity result on arbitrary UMD spaces. The weighted case appears to be new. Using Rubio de Francia extrapolation techniques one can extend Lemma \ref{lem:simpleestt} to a large class of Banach functions spaces $E(I,w;X)$ instead of $L^p(I,w;X)$ (see \cite{CMP}).
\end{remark}

\section{Maximal regularity for stochastic evolution equations}
\label{sec:MR}
In this section we consider the semilinear stochastic evolution equation
\begin{equation}\label{eq:SEE}
\left\{\begin{aligned}
dU(t)  +  A(t) U(t)\ud t & =  F(t,U(t)) \ud t  +  \big(B(t) U(t)  +G(t,U(t))\big) \ud W_H(t),\\
 U(0) & = u_0.
\end{aligned}
\right.
\end{equation}
Here $A(t)$ and $B(t)$ are linear operators which are $(t,\omega)$-dependent. The functions $F$ and $G$ are nonlinear perturbations.

In Subsections \ref{subsec:det} and \ref{subsec:SMR} we introduce the definitions of maximal $L^p$-regularity for deterministic equations and stochastic equations respectively. This extends well-known notions to the $(t,\omega)$-dependent setting. Moreover, we allow weights in time.
In Subsection \ref{subs:SMRreduction} we present a way to reduce the problem with time-dependent operators to the time-independent setting. In Subsection \ref{subs:semil} we show that if one has maximal $L^p$-regularity, then this implies well-posedness of semilinear initial value problems. Finally in Subsection \ref{subs:B=0} we explain a setting in which one can reduce to the case $B=0$.

\subsection{The deterministic case\label{subsec:det}}

Consider the following hypotheses.
\begin{assumption}\label{conditionX}
Let $X_0$ and $X_1$ be Banach spaces such that $X_1\hookrightarrow X_0$ is dense. Let $X_{\theta} = [X_0, X_1]_{\theta}$ and $X_{\theta,p} = (X_0, X_1)_{\theta,p}$  denote the complex and real interpolation spaces at $\theta\in (0,1)$ and $p\in [1, \infty]$, respectively.
\end{assumption}

For $f\in L^1(I;X_0)$ with $I = (0,T)$ and $T\in (0,\infty]$ we consider:
\begin{equation}\label{eq:deteq}
\left\{\begin{aligned}
u'(t) +  A(t) u(t) & =  f(t), \ \  t\in I\\
 u(0) & = 0.
\end{aligned}
\right.
\end{equation}
We say that $u$ is a {\em strong solution} of \eqref{eq:deteq} if for any finite interval $J\subseteq I$ we have $u\in L^1(J;X_1)$ and
\begin{equation}\label{eq:generalMR}
u(t)+\int_0^t A(s)u(s)\ud s= \int_0^t f(s) \ud s,\ \ t\in \overline{J},
\end{equation}
Note that this identity yields that $u\in W^{1,1}(J;X_0)$ and $u\in C(\overline{J};X_0)$ for bounded $J\subseteq I$.

\begin{definition}[Deterministic maximal regularity]
Let Assumption \ref{conditionX} be satisfied and assume that $A:[s,\infty)\to \calL(X_1, X_0)$ is strongly measurable and $\sup_{t\in \R}\|A(t)\|_{\calL(X_1, X_0)}<\infty$. Let $p\in (1, \infty)$, $\alpha\in (-1,p-1)$, $T\in (0,\infty]$, and set $I = (0,T)$. We say that $A\in \DMR(p,\alpha,T)$ if for all $f\in L^p(I,w_{\alpha};X_0)$, there exists a strong solution \[u\in W^{1,p}(I,w_{\alpha};X_0)\cap L^{p}(I,w_{\alpha};X_1)\]
of \eqref{eq:deteq} and
\begin{equation}\label{eq:MRest}
\|u\|_{W^{1,p}(I,w_{\alpha};X_0)} + \|u\|_{L^{p}(I,w_{\alpha};X_1)}\leq C\|f\|_{L^{p}(I,w_{\alpha};X_{0})}.
\end{equation}
\end{definition}

In \eqref{eq:generalMR} we use the continuous version of $u:\overline{I}\to X_0$.  By Proposition \ref{prop:traceembedding}  for $\alpha\in [0,p-1)$ we have
\[u\in C_{ub}(\overline{I};X_{1-\frac{1+\alpha}{p},p}) \ \ \text{and} \ \ u\in C_{ub}([\varepsilon,T];X_{1-\frac{1}{p},p}), \ \ \varepsilon\in (0,T).\]
If $\alpha\in (-1, 0)$ the first assertion does not hold, but the second one holds on $[0,T]$ if $T<\infty$.

\begin{remark}\label{rem:intervals}
Although we do allow $T=\infty$ in the above definition, most result will be formulated for $T\in (0,\infty)$  as this is often simpler and enough for applications to PDEs.

Note that $A\in \DMR(p,\alpha,T)$ implies that the solution $u$ is unique (use \eqref{eq:MRest}). Furthermore, it implies unique solvability of \eqref{eq:deteq} on subintervals $J = (a,b)\subseteq I$. In particular, $\DMR(p,\alpha, T)$ implies $\DMR(p,\alpha, t)$ for all $t\in (0,T]$.
\end{remark}

\subsection{Hypothesis on $A$ and $B$ and the definition of SMR\label{subsec:SMR}}

Consider the following hypotheses.
\begin{assumption}\label{conditionA} Let $H$ be a separable Hilbert space. Assume $X_0$ and $X_1$ are UMD spaces with type $2$. Let $A:\R_+\times\O\to \calL(X_1, X_0)$ be strongly progressively measurable and
    \[C_A:=\sup_{t\in \R,\omega\in \Omega}\|A(t,\omega)\|_{\calL(X_1, X_0)}<\infty.\]

Let $B:\R_+\times \O\to \calL(X_1, \calL(H,X_{\frac12}))$ be such that for all $x\in X_1$ and $h\in H$, $(Bx)h$ is strongly progressively measurable and assume there is a constant $C$ such that \[C_B:=\sup_{t\in \R,\omega\in \Omega}\|B(t,\omega)\|_{\calL(X_1,\calL(H,X_{\frac12}))}<\infty.\]
\end{assumption}

For $f\in L^1(I;X_0)$ and $g\in L^2(I;\gamma(H,X_{\frac12}))$ with $I = (0,T)$ and $T\in (0,\infty]$ we consider:
\begin{equation}\label{eq:stochcase}
\left\{\begin{aligned}
dU(t)  +  A(t) U(t)\ud t & =  f(t) \ud t  +  \big(B(t) U(t)  + g(t)\big) \ud W_H(t),\\
 U(0) & = 0.
\end{aligned}
\right.
\end{equation}
We say that $U$ is a {\em strong solution} of \eqref{eq:SEE} if for any finite interval $J\subseteq I$ we have $U \in L^0_{\F}(\Omega;L^2(J;\gamma(H,X_{1})))$ and
almost surely for all $t\in I$,
\begin{equation}\label{eq:generalMRStoch}
U(t)  +\int_0^t A(s)U(s)\ud s= \int_0^t f(s) \ud s + \int_0^t \Big(g(s) + B(s) U(s)\Big)\ud W_H(s),
\end{equation}
The above stochastic integrals are well-defined by \eqref{eq:type2est}.
Identity \eqref{eq:generalMRStoch} yields that $U$ has paths in $C(\overline{J};X_0)$ for bounded $J\subseteq I$ (see Lemma \ref{lem:simpleestt}).

\begin{definition}[Stochastic maximal regularity]
\label{def:SMR}
Suppose Assumptions \ref{conditionX} and \ref{conditionA} hold.
Let $p\in [2, \infty)$, $\alpha\in (-1,\frac{p}{2}-1)$ ($\alpha = 0$ is included if $p=2$), $T\in (0,\infty]$, and set $I = (0,T)$. We say that $(A,B)\in \SMR(p,\alpha,T)$ if for all $f\in L^p_{\F}(\Omega\times I,w_{\alpha};X_0)$ all $g\in L^p_{\F}(\Omega\times I,w_{\alpha};\gamma(H,X_{\frac12}))$, there exists a strong solution
\[U\in \bigcap_{\theta\in [0,\frac12)} L^p(\Omega;H^{\theta,p}(I,w_{\alpha};X_{1-\theta}))\]
of \eqref{eq:stochcase} and for each $\theta\in [0,\frac12)$ there is a constant $C_{\theta}$ such that
\begin{equation}\label{eq:SMRest}
\begin{aligned}
&\|U\|_{L^p(\Omega;H^{\theta,p}(I,w_{\alpha};X_{1-\theta}))} \\ &\qquad \leq
C_{\theta}\|f\|_{L^{p}(\Omega\times I,w_{\alpha};X_{0})} + C_{\theta}\|g\|_{L^{p}(\Omega\times I,w_{\alpha};\gamma(H,X_{\frac12}))}.
\end{aligned}
\end{equation}
In the case $B = 0$ we write $A\in \SMR(p,\alpha,T)$ instead of $(A,0)\in \SMR(p,\alpha,T)$
\end{definition}
In the above we use a pathwise continuous version of $U:\Omega\times \overline{I}\to X_0$. By Proposition \ref{prop:traceembedding} if $\alpha\in [0,\frac{p}{2}-1)$ we even have
\[U\in L^p(\Omega;C(\overline{I};X_{1-\frac{\alpha+1}{p},p})) \ \ \text{and} \ \ U\in L^p(\Omega;C([\varepsilon,T];X_{1-\frac{\alpha+1}{p},p})).\]
If $\alpha\in (-1, 0)$ the first assertion does not hold, but the second one holds on $[0,T]$ if $T<\infty$.

A variant of Remark \ref{rem:intervals} holds for $\SMR$. In particular, any of the estimates \eqref{eq:SMRest} implies uniqueness.
\begin{remark}
Unlike in the deterministic case the stochastic case does not allow for an optimal endpoint $H^{\frac12, p}$, because already a standard Brownian motion does not have paths in this space a.s. Therefore, we need to quantify over $\theta\in [0,\frac12)$ in the above definition.

In the case $-A$ is time-independent and generates an analytic semigroup, some different type of end-point results on the time-regularity in terms of Besov spaces have been obtained in \cite{OndrVer} which even include regularity at exponent $\frac12$ which is known to be the optimal regularity of a standard Brownian motion.
\end{remark}

In the time-independent case, many properties of $\DMR$ and $\SMR$ are known such as independence of $p$, $\alpha$ and $T$. For details we refer to \cite{Dore,PRSimweight} for the deterministic case and \cite{AgrVer,LoVer} for the stochastic case.

In the next two results we collect sufficient conditions for $\DMR$ and $\SMR$ in the time independent case. The first result follows from \cite[Theorem 5.3 and (3.6)]{KWcalc} and \cite[Theorem 4.2]{We} (in the latter $\DMR$ was characterized in terms of $R$-boundedness).

\begin{proposition}\label{prop:suffcondDMR}
Suppose Assumption \ref{conditionX} is satisfied and assume $X_0$ is a UMD space. Assume $A\in \calL(X_1, X_0)$.

If $A$ has a bounded $H^\infty$-calculus of angle $<\pi/2$ and $0\in\rho(A)$, then $A\in \DMR(p,\alpha, T)$ for all $p\in (1, \infty)$, $\alpha\in (-1,p-1)$ and $T\in(0,\infty]$.
\end{proposition}

In the time-independent setting the next result follows from \cite{NVW12a} for $\alpha= 0$ (also see \cite{NVWR,NVW15}). The case $\alpha\neq 0$ was obtained in \cite{AgrVer} by a perturbation argument.
\begin{proposition}\label{prop:suffcondSMR}
Suppose Assumption \ref{conditionX} is satisfied. Assume $A\in \calL(X_1, X_0)$.
Let $X_0$ be isomorphic to a $2$-convex Banach function space such that $(X^{1/2}_0)^*$ has the Hardy--Littlewood property (e.g. $X_0 = L^q(\dom;\ell^2)$, where and $q\in [2, \infty)$).

If $A$ has a bounded $H^\infty$-calculus of angle $<\pi/2$ and $0\in \rho(A)$, then $A\in \SMR(p,\alpha, T)$ for all $p\in (2, \infty)$, $\alpha\in (-1,\frac{p}{2}-1)$ and $T\in (0,\infty]$. Moreover, if $X_0$ is a Hilbert space, then the result in the case $(p,\alpha) = (2,0)$ holds as well.
\end{proposition}

\subsection{$\SMR$ for time-dependent problems}\label{subs:SMRreduction}
The next result is a useful tool to derive $A\in \SMR$ from $A\in \DMR$ and $A_0\in \SMR$ for a certain reference operator $A_0$ which one is free to choose. It extends \cite[Lemma 5.1]{Kry09VMO} and \cite{KimLeesystems} where the case with $A_0 = -\Delta$ on $X_0 = L^p$ with $\alpha=0$ was considered and where $A(t)$ was a second order operator.
\begin{theorem}\label{thm:timedepSMR}
Suppose Assumptions \ref{conditionX} and \ref{conditionA} hold.
Let $p\in [2, \infty)$, $\alpha\in (-1,\frac{p}{2}-1)$ ($\alpha=0$ if $p=2$ is allowed as well) and $T\in (0,\infty)$.
\begin{enumerate}[(i)]
\item\label{it:A0cond}
There exists a sectorial operator $-A_0$ with $D(A_0) = X_1$, and $X_{\frac12} = D((\lambda+A_0)^{1/2})$ such that $A_0\in \SMR(p,\alpha,T)$.
\item\label{it:detcond} Assume that there is a $C>0$ such that for all $\omega\in \O$, $A(\cdot,\omega)\in \DMR(p,\alpha,T)$ and \eqref{eq:MRest} holds with constant $C$.
\end{enumerate}
Then $A\in \SMR(p,\alpha, T)$.
\end{theorem}
\begin{proof}

In the proof we write
\[\MR_T := W^{1,p}((0,T), w_{\alpha};X_0)\cap L^p((0,T), w_{\alpha};X_1).\]
which we turn into a Banach space by using the sum norm.

{\em Step 1: Progressive measurability and estimates for the deterministic part}

Consider the mapping $\Lambda_T:\O\to \calL(\MR_T,L^p(0,T,w_{\alpha};X_0))$ given by $\Lambda_T(\omega) = d/dt + A(\cdot,\omega)$. Then $\Lambda_T$ is strongly $\F_T$-measurable and each $\Lambda_T(\omega)$ is invertible. It is well-known that its inverse mapping $\omega\mapsto \Lambda_T(\omega)^{-1}$ is strongly $\F_T$-measurable as well (see \cite{NaSa73}). For convenience we include a short argument for this special case. Fix $\omega_0\in \O$. Now $\omega \mapsto \Lambda_T(\omega)\Lambda_T(\omega_0)^{-1}\in \calL(\MR_T)$ is strongly $\F_T$-measurable and takes values in the invertible operators. Since taking inverses is a continuous mapping on the open set of invertible mappings it follows that  $\omega\mapsto \Lambda_T(\omega_{0})\Lambda_T(\omega)^{-1}$ is strongly $\F_T$-measurable as well. Clearly, the above holds with $T$ replaced by any $t\in (0,T]$ as well.

Now for $f\in L^p_{\F}(\Omega\times I;X_0)$, consider the problem
\[u' = A(t) u + f, \ \ u(0) = 0.\]
The solution is given by $u(\cdot,\omega) = \Lambda_T(\omega)^{-1} f(\cdot, \omega)$
and by \eqref{it:detcond}
\[\|u(\cdot,\omega)\|_{\MR_T} \leq C_1 \|f(\cdot,\omega)\|_{L^p((0,T),w_{\alpha};X_0)}.\]
Moreover, by the previous observations  $u$ is strongly $\F_T$-measurable (as an $\MR_T$-valued mapping) and we can take $L^p(\Omega)$-norms in the previous estimate to obtain
\begin{align}\label{eq:uLpOmegahulp}
\|u\|_{L^p(\Omega;\MR_T)} \leq C_1 \|f\|_{L^p(\Omega\times (0,T),w_{\alpha};X_0)}.
\end{align}
In the same way one can see that $u$ is progressively measurable. Indeed, for $t\in (0,T)$ $f|_{[0,t]}$ is strongly $\F_t\times \B([0,t])$-measurable and hence $u|_{[0,t]} = \Lambda_t^{-1}f|_{[0,t]}$ is strongly $\F_t$-measurable. By Proposition \ref{prop:funcspacesUMD}, we also have that $u \in SMR_{T}$
and that, for all $\theta \in [0,\frac{1}{2})$, $\|U\|_{SMR_{T,\theta}} \lesssim \|f\|_{L^p(\Omega\times (0,T),w_{\alpha};X_0)}.$

{\em Step 2: Main step}

It remains to prove existence and estimates in the space $\SMR_{T} :=\bigcap_{\theta\in [0,\frac12)} \SMR_{T,\theta}$, where
\[\SMR_{T,\theta} := L^p(\Omega; H^{\theta,p}(I,w_{\alpha};X_{1-\theta})).\]
Let $f\in L^p_{\F}(\Omega\times I,w_{\alpha};X_0)$ and $g\in L^p_{\F}(\Omega\times I,w_{\alpha};\gamma(H,X_{\frac12}))$. In order to prove $A\in \SMR(p,\alpha,T)$ consider
\begin{equation}\label{eq:SMRtoprove}
dU+AU\ud t = f\ud t + g\ud W,  \ \ U(0)= 0.
\end{equation}
We will build $U$ from the solutions of two sub-problems.

Since $A_0\in \SMR(p,\alpha,T)$ we can find $V_1\in \SMR_T$ such that
\[dV_1+A_0V_1\ud t =  g\ud W,  \ \ V_1(0)= 0\]
and one has the estimate, for each $\theta \in [0,\frac{1}{2})$,
\begin{equation}\label{eq:A0partest}
\|V_1\|_{\SMR_{T,\theta}}\leq C\|g\|_{L^p(\Omega\times I,w_{\alpha};\gamma(H,X_{\frac12}))}.
\end{equation}

By Step 1 we can find $V_2\in \SMR_T$ such that
\[V_2' + A V_2 = f + (A-A_0) V_1, \ \ V_2(0) = 0\]
and by \eqref{eq:uLpOmegahulp} and \eqref{eq:A0partest}
\begin{align*}
\|V_2\|_{\SMR_{T,\theta}}&\leq C \|f\|_{L^p(\Omega\times I,w_{\alpha};X_0)} + C\|(A-A_0) V_1\|_{L^p(\Omega\times I,w_{\alpha};X_0)}
\\ & \leq C \|f\|_{L^p(\Omega\times I,w_{\alpha};X_0)} + C\|g\|_{L^p(\Omega\times I,w_{\alpha};\gamma(H,X_{\frac12}))}.
\end{align*}

Now it is straightforward to check that $U = V_1+V_2$ is a solution to \eqref{eq:SMRtoprove} and combining the estimates for $V_1$ and $V_2$ we obtain $A\in \SMR(p,\alpha, T)$.
\end{proof}

The solvability of \eqref{eq:stochcase} with $B\neq 0$ can be a delicate matter. In particular it typically requires a stochastic parabolicity condition involving $A$ and $B$. However, there are several situations where one can prove an a priori estimate and where for a simple related problem one can prove existence and uniqueness of a solution in $L^p_{\F}(\Omega;L^p(I,w_{\alpha};X_1))$. These are the ingredients to apply the method of continuity (see \cite[Theorem 5.2]{GilTru}) to obtain existence and uniqueness of \eqref{eq:stochcase}. This is a well-known method, which we present in an abstract setting in the proposition below. For convenience let
\begin{align*}
E_{\theta} & = L^p_{\mathcal{F}}(\Omega; H^{\theta,p}(I,w_{\alpha};X_{1-\theta})),\\
Z_{\theta} & = L^p_{\mathcal{F}}(\Omega; L^{p}(I,w_{\alpha};X_{\theta})), \\
Z_{\theta}^{\gamma} & = L^p_{\mathcal{F}}(\Omega; L^{p}(I,w_{\alpha};\gamma(H,X_{\theta}))), \\
%Z^\gamma_{\theta} & = L^p_{\F}(\O;H^{\theta,p}(I,w_{\alpha};\g(H,X_{1-\theta}))),
\end{align*}
where $I = (0,T)$ with $T\in (0,\infty)$. The spaces $Z_{\theta}$ are the spaces in which the data is chosen. The spaces $E_{\theta}$ are the spaces in which the solution lives.

\begin{proposition}[Method of continuity]\label{prop:methodcont}
Suppose Assumptions \ref{conditionX} and \ref{conditionA} hold. Let $p\in [2, \infty)$, $\alpha\in [0,\frac{p}{2}-1)$, $\theta \in [0,\frac{1}{2})$, and $T\in (0,\infty)$, and set $I = (0,T)$. Let $\tilde{A}\in \calL(X_1, X_0)$ be given. For $\lambda\in [0,1]$ let
\[A_{\lambda}(t) = (1-\lambda)\tilde{A} + \lambda A(t), \ \ \text{and} \ \ B_{\lambda}(t) = \lambda B(t).\]
Consider the problem $U(0) = 0$ and
\begin{equation}\label{eq:stochcaselambda}
dU(t)  +  A_{\lambda}(t) U(t)\ud t  =  f(t) \ud t  +  \big(B_{\lambda}(t) U(t)  + g(t)\big) \ud W_H(t).
\end{equation}
\begin{enumerate}[(i)]
\item\label{it:methodconti} Assume that there is a constant $C$ such that, for all $\lambda\in [0,1]$, all $f\in Z_0$, and all $g\in Z_{\frac12}^{\gamma}$, any strong solution to \eqref{eq:stochcaselambda} $U\in E_{\theta}\cap E_{0}$ satisfies
\begin{equation}\label{eq:apropriori}
\|U\|_{E_{\theta}}+\|U\|_{E_{0}}\leq C( \|f\|_{Z_0} + \|g\|_{Z_{\frac12}^{\gamma}}).
\end{equation}
\item\label{it:methodcontii} Assume that, for all $f\in Z_0$ and all $g\in Z_{\frac12}^{\gamma}$, there exists a strong solution  $U\in E_{\theta} \cap E_{0}$ to \eqref{eq:stochcaselambda} with $\lambda = 0$.
\end{enumerate}
Then for all $\lambda\in [0,1]$, all $f\in Z_0$, and all $g\in Z_{\frac12}^{\gamma}$,  there exists a unique strong solution $U\in E_{\theta} \cap E_{0}$ of \eqref{eq:stochcaselambda}, and it satisfies the estimate \eqref{eq:apropriori}
\end{proposition}
In particular, the above result implies that if $(\tilde{A},0)\in \SMR(p,\alpha,T)$ and \eqref{it:methodconti} holds for all $\theta\in [0,\frac12)$, then $(A,B)\in\SMR(p,\alpha,T)$. Note that in \eqref{it:methodconti} we only assume that, as soon as a solution $U\in E_{\theta}\cap E_0$ to \eqref{eq:stochcaselambda} exists, then \eqref{eq:apropriori} holds.
\begin{proof}
The proof is a generalization of a standard method (see \cite[p.\ 218]{Kry}). We include the details for completeness. Note that uniqueness follows from \eqref{eq:apropriori}.
Let $\Lambda\subseteq [0,1]$ be the set of all points $\lambda$ such that for all $f\in Z_0$ and $g\in Z_{\frac12}^{\gamma}$ \eqref{eq:stochcaselambda} has a strong solution $U\in E_{\theta}$. It suffices to prove $1\in \Lambda$.
We claim that there exists an $\varepsilon>0$ such that for every $\lambda_0\in \Lambda$, $[\lambda_0-\varepsilon, \lambda_0+\varepsilon]\cap [0,1]\subseteq \Lambda$. Clearly, proving the claim would finish the proof.

To prove the claim let $\lambda_0 \in \Lambda$.
Fix $\lambda\in [\lambda_0-\varepsilon, \lambda_0+\varepsilon]\cap [0,1]$, where $\varepsilon>0$ is fixed for the moment. For $V\in E_{\theta}$, let $U\in E_{\theta}\cap E_{0}$ be the solution to
\begin{align*}
dU(t)  +  A_{\lambda_0}(t) U(t)\ud t  & =  [f(t)+(A_{\lambda_0}(t) - A_{\lambda}(t))V(t)]\ud t   \\ & \qquad + \big[B_{\lambda_0}(t) U(t) + g(t)+(B_{\lambda}(t) - B_{\lambda_0}(t)) V(t)\big] \ud W_H(t).
\end{align*}
In this case we write $L_{\lambda}(V) = U$. It is enough show that $L_{\lambda}:E_{\theta}\cap E_{0} \to FF_{\theta}\cap E_{0}$ is a strict contraction. Indeed, then by the Banach fixed point theorem there exists a unique $U\in E_{\theta}\cap E_{0}$ such that $L_{\lambda}(U) = U$ and this clearly implies that $U$ is a strong solution of \eqref{eq:stochcaselambda}.
To prove that $L$ is a strict contraction, let us note that for $V_1, V_2\in E_{\theta}\cap E_{0}$ and  $V = V_1-V_2$, the process $U = L_{\lambda}(V_1) - L_{\lambda}(V_2)$ is a strong solution to
\begin{align*}
dU(t)  +  A_{\lambda_0}(t) U(t)\ud t  & =  (A_{\lambda_0} - A_{\lambda}(t))V(t)\ud t   \\ & \ + \big[B_{\lambda_0}(t) U(t) +(B_{\lambda}(t) - B_{\lambda_0}(t)) V(t)\big] \ud W_H(t).
\end{align*}
Therefore, by \eqref{eq:apropriori}
\begin{align*}
\|L_{\lambda}(V_1)-L_{\lambda}(V_2)\|_{E_{\theta}\cap E_{0}} & \leq C \|(A_{\lambda_0} - A_{\lambda})V\|_{Z_0} + C\|(B_{\lambda_0} - B_{\lambda})V\|_{Z_\frac12^{\gamma}}
\\ & \leq \tilde{C}\varepsilon \|V_1 - V_2\|_{E_{0}}
\leq \tilde{C}\varepsilon \|V_1 - V_2\|_{E_{\theta \cap E_{0}}},
\end{align*}
where $\tilde{C} = C (C_A +\|\tilde{A}\|+C_B)$. Here we used
\begin{align*}
\|(A_{\lambda_0} - A_{\lambda})u\|_{X_0} & \leq |\lambda_0-\lambda| (C_A + \|\tilde{A}\|)\|u\|_{X_1},
\\ \|(B_{\lambda_0} - B_{\lambda})u\|_{X_{\frac12}} & \leq |\lambda_0-\lambda| C_B\|u\|_{X_1}.
\end{align*}
Therefore, letting $\varepsilon  = \frac{1}{2C (C_A +\|\tilde{A}\|+C_B)}$ we see that $L_{\lambda}$ is a strict contraction.
\end{proof}

\subsection{Semilinear equations}\label{subs:semil}
In this section we show that our maximal regularity set-up allows for simple perturbation arguments in order to include nonzero initial values and nonlinear functions $F$ and $G$ as in \eqref{eq:SEE} on a fixed time interval $I = (0,T)$ as soon as one knows that $(A,B)\in \SMR(p,\alpha,T)$. The results extend \cite[Theorems 4.5 and 5.6]{NVWsiam} to a setting where we only assume measurability in time and where we can take {\em rough} initial values.

Consider the following conditions on $A$:
\begin{assumption}\label{conditionMR}
Suppose Assumptions \ref{conditionX} and \ref{conditionA} hold. Assume $(A,B)\in \SMR(p,\alpha,T)$ and let $K_{\rm det}$ and $K_{\rm st}$ be such that
the strong solution to \eqref{eq:stochcase} satisfies
\[\|U\|_{L^p(\Omega;L^p(I,w_{\alpha};X_1))}\leq K_{\rm det} \|f\|_{L^p(\Omega;L^p(I,w_{\alpha};X_0))} + K_{\rm st} \|g\|_{L^p(\Omega;L^p(I,w_{\alpha};\gamma(H,X_{\frac12})))}.\]
\end{assumption}
Note that the constants $K_{\rm det}$ and $K_{\rm st}$ exists by the condition $(A,B)\in \SMR(p,\alpha,T)$. We introduce them in order to have more explicit bounds below.

Consider the following conditions on $F$ and $G$.
\begin{assumption}\label{as:LipschitzF}
The function $F:[0,T]\times\O\times X_1\to X_0$ is strongly
progressively measurable, $F(\cdot, \cdot,0)\in L^p(\Omega;L^p(I,w_{\alpha});X_0)$,
and there exist $L_{F}$ and $\tilde L_{F}$ such that for all $t\in [0,T]$,
$\omega\in \O$, and $x,y\in X_1$,
\begin{equation*}
\phantom{aaaaa}
\|F(t,\omega, x) - F(t,\omega,y)\|_{X_0} \leq L_{F} \|x-y\|_{X_1} +
\tilde L_{F} \|x-y\|_{X_0}
\end{equation*}
\end{assumption}

\begin{assumption}
\label{as:LipschitzG}
The function $G:[0,T]\times\O\times X_1\to \g(H,X_{\frac12})$
is strongly progressive measurable, $G(\cdot, \cdot, 0)\in L^p(\Omega;L^p(I,w_{\alpha});\g(H,X_{\frac12}))$ and there exist $L_{G}$, $\tilde L_{G}$ such that for all $t\in
[0,T]$, $\omega\in \O$, and $x,y\in X_1$,
\begin{equation*}
\phantom{aaaaa}
\|G(t,\omega, x) - G(t,\omega,y)\|_{\g(H,X_{\frac12})} \leq L_{G} \|x-y\|_{X_1} +
\tilde L_{G} \|x-y\|_{X_0}.
\end{equation*}
\end{assumption}

\begin{definition}\label{def:strongsol}
Suppose Assumptions \ref{conditionX}, \ref{conditionA}, \ref{as:LipschitzF} and \ref{as:LipschitzG} are satisfied. Let $u_0:\Omega\to X_0$ be strongly $\F_0$-measurable.
A process $U: [0,T]\times\Omega \to X_0$ is called a
{\em strong solution} of \eqref{eq:SEE} if it is strongly progressively measurable,
and
\begin{enumerate}[(i)]
\item almost surely, $U\in L^2(0,T;X_1)$;

\item almost surely for all $t\in [0,T]$, the following identity holds in $X_0$:
\begin{align*}
\phantom{aa}
U(t) + \int_0^t A(s) U(s) \, ds = u_0 & + \int_0^t F(s,U(s)) \, ds \\ & + \int_0^t
\Big(B(s)U(s) +G(s,U(s))\Big) \, d W_H(s).
\end{align*}
\end{enumerate}
\end{definition}
It is straightforward to check that all integrals are well-defined by the assumptions.

Now we state the main result of this subsection:
\begin{theorem}\label{thm:nonlinearperturbation}
Let $p\in [2, \infty)$ and $\alpha\in [0, \frac{p}{2}-1)$ ($\alpha =0$ is allowed if $p=2$) and $T\in (0,\infty)$. Set $\delta = 1-\frac{\alpha+1}{p}$ and $I = (0,T)$.
Consider the following conditions:
\begin{enumerate}[$(1)$]
\item Suppose Assumptions \ref{conditionMR}, \ref{as:LipschitzF}, \ref{as:LipschitzG} hold,  and $u_0\in L^p(\Omega,\F_0;X_{\delta,p})$.
\item Assume the Lipschitz constants $L_F$ and $L_G$
satisfy
$$K_{\rm det} L_F  + K_{\rm st} L_G <1.$$
\item There exists a sectorial operator $A_0$ on $X_0$ with $D(A_0)=X_1$ and angle $<\pi/2$.
\end{enumerate}
Then the following assertions hold:

Problem \eqref{eq:SEE} has a unique strong solution $U\in L^p_{\F}(\O;L^p(I,w_{\alpha};X_1))$.
Moreover, there exist constants $C, C_{\varepsilon}, C_{\varepsilon,\theta}$ depending on $X_0, X_1, p, \alpha, T,A,B,A_0, K_{\rm det}, K_{\rm st}$ and the Lipschitz constants of $f$ and $g$ such that
\begin{align}
\label{eq:nonlinperReg1a}
\|U\|_{L^p(\O;C(\overline{I}; X_{\delta,p}))}& \leq C K_{u_0,F,G},
\\ \label{eq:nonlinperReg1b}
\|U\|_{L^p(\O;C([\varepsilon,T]; X_{1-\frac1p,p}))}& \leq C_{\varepsilon} K_{u_0,F,G}, \ \ \varepsilon\in (0,T].
\\ \label{eq:nonlinperReg2} \|U\|_{L^p(\Omega;H^{\theta,p}(I,w_{\alpha};X_{1-\theta}))} &\leq C_{\theta}K_{u_0,F,G}, \ \ \theta\in [0,\tfrac12)
\\ \label{eq:nonlinperReg3} \|U\|_{L^p(\Omega;C^{\theta-\frac{1+\alpha}{p}}(\overline{I};X_{1-\theta}))}& \leq C_{\theta}K_{u_0,F,G}, \ \ \theta\in (\tfrac{1+\alpha}{p},\tfrac12)
\\ \label{eq:nonlinperReg4} \|U\|_{L^p(\Omega;C^{\theta-\frac{1}{p}}([\varepsilon,T];X_{1-\theta}))}&\leq C_{\varepsilon,\theta}K_{u_0,F,G}, \ \ \theta\in (\tfrac1p,\tfrac12), \varepsilon\in (0,T].
\end{align}
where
\begin{equation}\label{eq:Ku0FG}
\begin{aligned}
K_{u_0,F,G} &= \|u_0\|_{L^p(\Omega;X_{\delta,p})}+\|F(\cdot, \cdot, 0)\|_{L^p(\Omega;L^p(I,w_{\alpha},X_0))} \\ & \qquad + \|G(\cdot, \cdot, 0)\|_{L^p(\Omega;L^p(I,w_{\alpha},\gamma(H,X_{\frac12})))}.
\end{aligned}
\end{equation}
Furthermore, if $U^{1}, U^2$ are  the strong solution of \eqref{eq:SEE} with initial value $u_0^1,u_0^2\in L^p(\Omega,\F_0;X_{\delta,p})$ respectively, then each of the above estimates holds with $U$ replaced by $U^1-U^2$, and $K_{u_0,F,G}$ replaced by $K_{u_0^1-u_0^2,F,G}$ on the right-hand side.
\end{theorem}
\begin{proof}
In the proof we use a variation of the arguments in \cite[Theorems 4.5]{NVWsiam}. Let us assume, without loss of generality that $L_{F},L_{G} \neq 0$, and that
$K_{\rm det} L_F  + K_{\rm st} L_G = 1-\nu$ for some $\nu\in (0,1)$.

{\em Step 0:} Reduction to $u_0 = 0$.
We consider $\Phi: t \mapsto e^{-tA_0} u_0$. Since $u_{0} \in X_{\delta,p}$, we have, by \cite[1.14.5]{Tr1}, that
\[\|\Phi\|_{L^p(I,w_{\alpha};X_1)} + \|\Phi\|_{W^{1,p}(I,w_{\alpha};X_0)} \leq C \|u_0\|_{X_{\delta,p}}.\]
Moreover, $\Phi$ is strongly progressively measurable, and Proposition \ref{prop:funcspacesUMD} gives that \begin{equation}\label{eq:initial}
\|\Phi\|_{L^p(\Omega,H^{\theta,p}(I,w_{\alpha};X_{1-\theta}))}\leq C \|u_0\|_{L^p(\Omega,X_{\delta,p})},
\end{equation}
for all $\theta\in [0,1]$.

Similarly, $\Phi$ can be estimated in all of the norms used in \eqref{eq:nonlinperReg1a}-\eqref{eq:nonlinperReg4}, by a constant multiple of
$\|u_0\|_{L^p(\Omega,X_{\delta,p})}$.
The process $V := U - \Phi$ is then a solution of
\[dV(t) + A(t)V(t) \ud t = \tilde{F}(t,V(t))\ud t +  [B(t) V(t) + \tilde{G}(t,V(t))]\ud W_H(t), \ \ V(0) = 0 \]
where $\tilde{F}(t,x) = F(t,x+\Phi(t))$ and
$\tilde{G}(t,x) = B(t) \Phi(t) + G(t,x+\Phi(t))$ satisfy the same conditions as $F$ and $G$. This completes the reduction to $u_{0}=0$.

\smallskip
{\em Step 1:} Local existence and uniqueness.
We first prove existence in spaces of the form
\begin{align*}
Z_{\theta,\kappa} & = L^p_{\F}(\O;L^p((0,\kappa),w_{\alpha};X_\theta)),\\
Z^\gamma_{\theta,\kappa} & = L^p_{\F}(\O;L^p((0,\kappa),w_{\alpha};\g(H,X_{\theta}))),
\end{align*}
where $\kappa\in (0,T)$ will be determined later, and $\theta\in [0,1]$.
To simplify notation, we omit the parameter $\kappa$, and consider the
norm $\nnn\cdot\nnn$ on $Z_{1}$, defined by
\[\nnn\phi\nnn = \|\phi\|_{Z_{1}} + M  \|\phi\|_{Z_{0}}\]
with $M = (1-\nu)^{-1}(K_{\rm det} \tilde L_{F} + K_{\rm st} \tilde L_{B})$.

For $\phi\in Z_1$ we consider the linearised problem
\begin{equation}\label{eq:SEEfixed}
\left\{\begin{aligned}
dU(t)  +  A(t) U(t)\ud t & =  F(t,\phi(t)) \ud t  +  \big(B(t) U(t)  +G(t,\phi(t))\big) \ud W_H(t),\\
 U(0) & = 0.
\end{aligned}
\right.
\end{equation}
By Assumptions \ref{as:LipschitzF} and \ref{as:LipschitzG}, we have that $F(\cdot, \phi)\in L^p_{\F}(\Omega;L^p(I,w_{\alpha};X_0))$ and $G(\cdot, \phi)\in L^p_{\F}(\Omega;L^p(I,w_{\alpha};\gamma(H,X_{\frac12}))$. Therefore, by Assumption \ref{conditionMR},  there exists a bounded map $L:Z_1\to Z_1$ such that $L(\phi)$ is the (unique) strong solution of \eqref{eq:SEEfixed}.

By linearity, we thus have that, for $\phi_1,\phi_2\in Z_1$, the process $U = L(\phi_1)-L(\phi_2)$ is a strong solution of
\begin{align}\label{eq:hulpfixed}
dU(t)  +  A(t) U(t)\ud t & =  f(t) \ud t  +  \big(B(t) U(t)  +g(t)\big) \ud W_H(t), \ \ u(0) = 0,
\end{align}
where $f = F(\cdot,\phi_1) - F(\cdot,\phi_2)$ and $g = G(\cdot,\phi_1) - G(\cdot,\phi_2)$.
Therefore, by Assumptions \ref{conditionMR}, \ref{as:LipschitzF} and \ref{as:LipschitzG},
\begin{align*}
\|L(\phi_1)-L(\phi_2)\|_{Z_1} & = \|U\|_{Z_1}
\\ & \leq K_{\rm det} \|f\|_{Z_0} + K_{\rm st} \|g\|_{Z_{\frac12}^\g}
\\ & = K_{\rm det} \|F(\cdot, \phi_1)-F(\cdot,
\phi_2)\|_{Z_0} + K_{\rm st} \|G(\cdot, \phi_1) - G(\cdot,
\phi_2)\|_{Z_{\frac12}^\g}
\\ & \leq K_{\rm det} L_F \|\phi_1-\phi_2\|_{Z_1} + K_{\rm det} \tilde L_{F}
\|\phi_1-\phi_2\|_{Z_0} \\ & \qquad   + K_{\rm st} L_G \|\phi_1 -
\phi_2\|_{Z_1} + K_{\rm st} \tilde L_{G} \|\phi_1 - \phi_2\|_{Z_{0}}
\end{align*}
We thus have that
\begin{equation}
\label{eq:Z1est}
\|L(\phi_1)-L(\phi_2)\|_{Z_1}\leq  (1-\nu) \nnn \phi_1-\phi_2\nnn,
\end{equation}
by definition of $\nu$ and $\nnn \cdot\nnn$.

Moreover, for fixed $t\in [0,\kappa]$, we have that
\begin{align*}
\|L(\phi_1)(t)-L(\phi_2)(t)\|_{L^p(\Omega;X_0)} & = \|U(t)\|_{L^p(\Omega;X_0)}
\\ & \leq C_A C_{\kappa,1} \|U\|_{Z_{1}} + C_{\kappa,1}\|f\|_{Z_{0}}  \\ & \qquad +  C_B C_X C_{\kappa,1}\|U\|_{Z_{1}}  + C_{\kappa,1} C_X \|g\|_{Z_{\frac12}^{\gamma}},
\end{align*}
using \eqref{eq:hulpfixed} and the easy part of Lemma \ref{lem:simpleestt}. Here $C_X$ denotes the embedding constant of $X_{1/2}\hookrightarrow X_0$, and $C_A, C_{B}$ are the constants in Assumption \ref{conditionA}.
Then, using first H\"older's inequality, Assumptions \ref{as:LipschitzF} and \ref{as:LipschitzG}, and estimate \eqref{eq:Z1est} next, we have that
\begin{align*}
\|L(\phi_1)(t)-L(\phi_2)(t)\|_{L^p(\Omega;X_0)}
& {\leq} C_{\kappa,2} \|U\|_{Z_1} + C_{\kappa,2} \|\phi_1-\phi_2\|_{Z_1},
\\ & {\leq} C_{\kappa} \|\phi_1-\phi_2\|_{Z_1}.
\end{align*}

Taking $L^p((0,\kappa), w_{\alpha})$-norms in $t$ on both sides we obtain
\[\|L(\phi_1)-L(\phi_2)\|_{Z_0}\leq c(\kappa)\nnn \phi_1-\phi_2\nnn,\]
and thus
\[\nnn L(\phi_1)-L(\phi_2) \nnn \leq (1-\nu + Mc(\kappa)) \nnn \phi_1-\phi_2\nnn,\]
with $\lim_{\kappa \downarrow 0}c(\kappa) = 0$.
Setting
\[\kappa : = \inf\{t\in (0,T]: \, Mc(t) \geq \tfrac12\nu\},\]
(or $\kappa = T$ if the infimum is taken over the empty set), we have that $(1-\nu + Mc(\kappa))\leq 1-\tfrac12\nu$,
and therefore that $L$ has a unique fixed point $U \in Z_1$.
The end time $\kappa$ only depends on $\nu$, $p$, $\alpha$, the constants $C_A, C_B, K_{\rm det}, K_{\rm st}$, the Lipschitz constants of $F$ and $B$,  and the spaces $X_0$ and $X_1$.

Considering a version of  $U$ with continuous paths (see comment below \eqref{eq:generalMRStoch}), we can assume that, for all  $t\in [0,\kappa]$, $U(t)=L(U(t))$ holds almost surely. The process $U$ is thus the unique strong solution of \eqref{eq:SEE}, and satisfies
\[\|U\|_{Z_{1}} = \|L(U)\|_{Z_{1}}\leq \|L(U) - L(0)\|_{Z_{1}} + \|L(0)\|_{Z_{1}}\leq (1-\tfrac12 \nu) \|U\|_{Z_{1}} + K_{u_0, F,G},\]
which gives
\begin{align*}
\|U\|_{Z_{1,\kappa}} \leq C K_{u_0,F,G},
\end{align*}

\smallskip
{\em Step 2: Regularity}.
Let $S \in (0,T)$, and $U$ be a strong solution of \eqref{eq:SEE} on the time interval $J=[0,S]$. Assume that
\begin{align}\label{eq:hulpmaxregLp}
\|U\|_{Z_{1,S}} \leq C K_{u_0,F,G}.
\end{align}
Then, from Step $0$ and $(A,B)\in\SMR(p, \alpha, S)$, we obtain, for $U=L(U)$:
\begin{align*}
\|U\|_{L^p(\Omega;H^{\theta,p}(J,w_{\alpha};X_{1-\theta}))}  & \leq C\|u_0\|_{L^p(\Omega;X_{\delta,p})}  + C\|F(\cdot, U)\|_{Z_{0,S}}  + C\|G(\cdot, U)\|_{Z_{\frac12,S}^{\gamma}}
\\ & \leq CK_{u_0, F, G} + C_{F,G}\|U\|_{Z_{1,S}}
\\ & \leq C\tilde{K}_{u_0, F, G},
\end{align*}
for all $\theta\in [0,\frac12)$, using Assumptions \ref{as:LipschitzF}, \ref{as:LipschitzG} and \eqref{eq:hulpmaxregLp}.  This proves \eqref{eq:nonlinperReg2} for $I = J$. Thus \eqref{eq:nonlinperReg3} and \eqref{eq:nonlinperReg4} for $I = J$ follows from Proposition \ref{prop:Sobolevemb}. Finally, \eqref{eq:nonlinperReg1a} and \eqref{eq:nonlinperReg1b} follow from Proposition \ref{prop:traceembedding}.

\smallskip
{\em Step 3: Global existence and uniqueness}.

To prove global existence and uniqueness let $S \in (0,T)$, and $U \in Z_{1,S}$ be a strong solution of \eqref{eq:SEE} on the time interval $J=[0,S]$. To obtain global existence, we just have to show that there exist an $\eta>0$ (independent of $S$) and a strong solution $U \in Z_{1,S+\eta}$  on the interval $[S, S+\eta]$, with initial condition $U(S)$ at time $S$.
By Step 2, we have that, for every $\varepsilon\in (0,S)$,
$$\|U\|_{L^p(\O;C([\varepsilon,S]; X_{1-\frac1p,p}))} \leq C_{\varepsilon} K_{u_0,F,G}.$$
We can thus define $V(t) = U - e^{-(t-S)A_0} U(S)$ and, reduce the problem to $V(S)=0$, as in Step 0.
Repeating Step 1, we find an $\eta>0$ (depending on the parameters only as $\nu$ did) and a unique strong solution $U\in L^p(\Omega;L^p((S, S+\eta);X_1))$. The regularity estimates \eqref{eq:nonlinperReg1a}-\eqref{eq:nonlinperReg4} then follow from Step 2, and global existence and uniqueness is proven by repeating this procedure finitely many times.

\smallskip
{\em Step 4: Continuous dependence}.
For $\kappa$ as in Step 1, the process $U = U^1 - U^2$ is a strong solution of
\begin{align*}
dU(t)  +  A(t) U(t)\ud t & =  f(t) \ud t  +  \big(B(t) U(t)  +g(t))\big) \ud W_H(t), \ \ u(0) = u_0^1-u_0^2,
\end{align*}
with $f = F(\cdot,U^1) - F(\cdot,U^2)$ and $g = F(\cdot,U^1) - F(\cdot,U^2)$.
Repeating Step 1, we have that
\[\nnn U^1 - U^2\nnn \leq (1-\frac12\nu) \nnn U^1 - U^2\nnn + \|u_0^1 - u_0^2\|_{L^p(\Omega;X_{\delta,p})},\]
and thus
\[\|U^1- U^2\|_{Z_{1,\kappa}} \leq C\|u_0^1 - u_0^2\|_{L^p(\Omega;X_{\delta,p})}.\]
Step 2 then gives the regularity estimates, while Step 3 extends the result from $[0,\kappa]$ to $[0,T]$, which concludes the proof.
\end{proof}

\begin{remark}
The results of Theorem \ref{thm:nonlinearperturbation} can be further ``localized'' to include non-integrable initial values, and locally Lipschitz functions $F$ and $G$. We refer to \cite[Theorem 5.6]{NVWsiam} for details.
\end{remark}

\subsection{Reduction to $B = 0$}\label{subs:B=0}
Before we continue to applications to SPDEs we show that there is a setting in which one can reduce to the case where $B = 0$ by using It\^o's formula. Such a reduction is standard
(cf.\ \cite[Theorem 3.1]{BNVW}, \cite[Section 6.6]{DPZ} and \cite[Section 4.2]{Kry}), but it seems that the general setting below has never been considered before. It leads to an abstract form of a so-called stochastic parabolicity condition. In the variational setting ($p=2$) a stochastic parabolicity condition appears in a more natural way (see \cite{LiuRock,Rozov}). We refer to \cite{BrzVer,DuLiuZhang,KimLeesystems,KryLot} for situations in which one cannot reduce to $B = 0$, but where one still is able to introduce a natural $p$-dependent
stochastic parabolicity condition.

\begin{assumption}
\label{conditionB}
Let $B:\R_+\times \O\to \calL(X_1, \calL(H,X_{\frac12}))$ be given by
\begin{align}\label{eq:defB}
B(t)x = \sum_{j=1}^J b_j(t) \otimes B_jx
\end{align}
Here $b_j:\R_+\times\O\to H$ is progressively measurable and there is a constant $M_0\geq 0$ such that for almost all $(t,\omega)$, $\|b_j(t,\omega)\|_H\leq M_0$.
Each of the operators $B_j$ generates a strongly continuous group on $X_k$ for $k\in \{0,1\}$ and there exists an $M_1\geq 0$ such that for all $t\in \R$, $\|e^{tB_j}\|_{\calL(X_k)}\leq M_1$ for $k\in \{0,1\}$.
For every $i,j\in\{1, \ldots, J\}$ and $s,t\in \R$, $e^{sB_i}$ and $e^{tB_j}$ commute on $X_0$, and $e^{sB_i}$ and $A(t)$  commute on $X_1$. Furthermore assume $X_{\frac12} \subseteq  D(B_j)$ and $X_1 \subseteq D(B_j^2)$.

The adjoints $A(t,\omega)^*$ are closed operators on $X_0^*$ and have a constant domain $D_{A^*}$ such that $D_{A^*}\subseteq D((B_j^*)^2)$.
\end{assumption}

Let $\alpha \in [0,1]$, $S_B:\R^J\to \calL(X_\alpha)$ and $\zeta:\R_+\times\O\to \R^J$ be given by
\begin{align}\label{eq:SBazeta}
S_B(a) = \exp\Big(\sum_{j=1}^J a_j B_j\Big), \qquad \text{and} \qquad \zeta_j(t) = \int_0^t b_j \ud W_H.
\end{align}

Consider the problem
\begin{equation}\label{eq:tildeUeq}
\begin{aligned}
\ud \wt{U}(t) + \wt{A}(t) \wt{U}(t) \ud t & = \Big(\wt{f}(t) - [B(t), \wt{g}(t)]\Big) \ud t  +  \wt{g}(t) \ud  W_H(t)
\\ \wt{U}(0) & = u_0,
\end{aligned}
\end{equation}
where $\wt{f}(t) = S_B(-\zeta(t)) f(t)$ and $\wt{g}(t) = S_B(-\zeta(t)) g(t)$ and $\wt{A}(t) = A(t) + \frac12 [B(t), B(t)]$ with
\begin{align}\label{eq:BBoperator}
[B(t), B(t)] = \sum_{i,j=1}^J (b_i(t), b_j(t))_H B_i B_j, \ \ \ [B(t), \wt{g}(t)] = \sum_{j=1}^J b_j(t) B_j \wt g(t)
\end{align}
Usually $[B(t), B(t)]$ is of ``negative'' type, whereas $A(t)$ is of ``positive'' type

Next we show that the problems \eqref{eq:stochcase} and \eqref{eq:tildeUeq} are equivalent under the above commutation conditions.
\begin{theorem}\label{thm:redB0}
Suppose Assumptions \ref{conditionX}, \ref{conditionA} and \ref{conditionB} hold.  Let $\wt U:[0,T]\times\O\to \gamma(H,X_{1})$ be progressively measurable and assume $\wt U\in L^2(0,T;\gamma(H,X_{1})$ a.s.\ Let ${U}(t) = S_B(\zeta(t)) \wt U(t)$, where $S_B$ and $\zeta$ are as in \eqref{eq:SBazeta}. Then $U$ is a strong solution to \eqref{eq:stochcase} if and only if $\wt{U}$ is a strong solution to \eqref{eq:tildeUeq}.
Moreover, $(A,B)\in \SMR(p,\alpha,T)$ if and only if $(\tilde{A}, 0)\in \SMR(p,\alpha,T)$.
\end{theorem}

\begin{proof}
Fix $\psi\in D_{A^*}$ and let $\phi:\R^J\to D_{A^*}$ be given by $\phi(a) = S_B(a)^*\psi$.
First assume $\wt{U}$ is a strong solution to \eqref{eq:tildeUeq}.
The aim of this first step it to apply It\^o calculus to find a formula for $\lb U, \psi\rb = \lb \wt{U}, \phi(\zeta)\rb$.

As a first step we apply It\^o's formula to $\phi(\zeta)$.
For this note that the operators $B_j^*$ are commuting as well. Moreover, one can check that $e^{sB_j^*}$ leaves $D_{A^*}$ invariant and since $D_{A^*}\subset D((B_j^*)^2)$, it follows that $\phi$ is twice continuously differentiable and
\[(\nabla \phi(a))_j  = B_j^* \phi(a),  \ \ \ (\nabla^2\phi(a))_{i,j} = B_i^* B_j^* \phi(a).\]
By It\^o's formula (see \cite{BNVW}) a.s.\ for all $t\in [0,T]$,
\[\phi(\zeta(t)) - \psi = \int_0^t B(s)^*  \phi(\zeta(s)) \ud W_H(s) + \frac{1}{2} \int_0^t [B(s),B(s)]^* \phi(\zeta(s))  \ud s,\]
where $[B(s),B(s)]^*$ stands for the adjoint of $[B(s),B(s)]$ (see \eqref{eq:BBoperator}).

Now applying It\^o's formula to the duality pairing on $X\times X^*$ gives a.s.\ for all $t\in [0,T]$
\begin{align*}
\lb U(t), \psi\rb - \lb u_0, \psi\rb & = \lb \wt{U}(t), \phi(\zeta(t))\rb - \lb u_0, \psi\rb
\\ & = \int_0^t -\lb \wt{A}(s) \wt{U}(s), \phi(\zeta(s))\rb + \lb \wt{f}(s) - [B(s), \wt{g}(s)], \phi(\zeta(s))\rb \ud s
\\ & \quad + \frac12 \int_0^t \lb \wt{U}(s), [B(s),B(s)]^* \phi(\zeta(s))\rb \ud s
\\ & \quad + \int_0^t \wt{g}(s)^* \phi(\zeta(s)) \ud W_H(s) + \int_0^t \lb \wt{U}(s), B(s)^* \phi(\zeta(s))\rb \ud W_H(s)
\\ & \quad + \int_0^t   \lb [B(s), \wt{g}(s)], \phi(\zeta(s))\rb  \ud s
\\ & = \int_0^t -\lb A(s) U(s), \psi\rb + \lb f(s), \psi\rb \ud s
\\ & \quad + \int_0^t g(s)^* \psi \ud W_H(s) + \int_0^t (B(s) U(s))^* \psi \ud W_H(s),
\end{align*}
where we used Assumption \eqref{conditionB} to commute $S_B(a)$ with $B_j$ and $A(s)$. By Hahn--Banach it follows that $U$ is strong solution of \eqref{eq:SEE}.

Similarly, if $U$ is a strong solution to \eqref{eq:SEE} one sees that $\wt{U}$ is a strong solution to \eqref{eq:tildeUeq} by applying It\^o's formula to
$\lb \wt{U}, \phi(\zeta)\rb$, where now $\phi:\R^J\to \Psi$ is given by $\phi(a) = S_B(-a)^*\psi$.

The final assertion is clear from the properties of $S_B(\pm \zeta(t))$.
\end{proof}

\section{Parabolic systems of SPDEs of $2m$-th order}\label{sec:2morder}

In this section we develop an $L^p(L^q)$-theory for systems of SPDEs of order $2m$. The case $m=1$ (where more can be proven) will be considered in Section \ref{sec:secondorder}. A similar setting was considered in \cite[Section 6]{NVWsiam} but under a regularity assumption on the coefficients of the operator $A$.  At the same time it extends some of the results in \cite{Kry}. For more discussion on this we refer to Section \ref{sec:secondorder}.
Finally we mention that the temporal weights allow us to obtain an $L^p(L^q)$-theory for a wider class of initial values than usually considered.

The main novelty in the results below are that the coefficients $a_{\alpha\beta}$ are allowed to be matrix-valued, random, and in the time-variable we do not assume any smoothness of the coefficients. The precise assumptions are stated below.

In  this section we consider the following system of stochastic PDEs on $[0,T]\times \R^d$:
\begin{equation}\label{eq:stochsystem}
\left\{
  \begin{array}{ll}
  dU(t)+ A(t) U(t)\ud t & = f(t,U(t))\ud t +  \sum_{n\geq 1} g_n(t,U(t)) \ud w_n(t), \\
    U(0) & = u_0,
  \end{array}
\right.
\end{equation}
where $w_n$ is a sequence of independent standard Brownian motions. The function $U:[0,T]\times \Omega\to L^p(\R^d;\C^N)$ is the unknown.

The operator $A$ is given by
\begin{equation*}
(A(t,\omega) \phi)(x) = (-1)^m \sum_{|\alpha|, |\beta|=m}  a_{\alpha\beta}(t,\omega,x)D^{\alpha} D^\beta \phi(x), \ \ x\in \R^d, t\in [0,T], \omega\in \Omega.
\end{equation*}
and there will be no need to consider lower order terms as they can be absorbed in the function $f$.

\begin{remark}
\label{rk:div form}
We work in the non-divergence form case, but the divergence form case can be treated in the same manner. Indeed, we decompose the problem into a time-dependent deterministic part, and a time-independent stochastic part. We thus only need to use \cite[Theorem 6.3]{DKnew} instead of \cite[Theorem 5.2]{DKnew} for the time-dependent deterministic part.
\end{remark}

\begin{remark}\label{}
In \eqref{eq:stochsystem} we have left out the $B$-term which we did consider in \eqref{eq:SEE} in the abstract setting. However, any operator $B(t)$ of $(m-1)$-th order with coefficients which are $W^{1,\infty}$ in the space variable could be included as well. Moreover, no stochastic parabolicity is required as $B$ is of lower order. We choose to leave this out as one can insert this into the nonlinearities $g_n$. If one wishes to include operators $B$ which are $m$-th order, then this either requires a smallness condition in terms of the maximal regularity estimates, or in order to apply Section \ref{subs:B=0} the highest order terms needs to be a group generator.
It follows from \cite[Theorem 0.1]{Brenner73} (and from \cite[Theorem 1.14]{Horm60} if $d=N=1$) that if $B$ is a differential operator of order $\geq 2$ and $B$ generates a strongly continuous group then necessarily $q=2$.
\end{remark}

In order to state the conditions on the coefficients $a_{\alpha\beta}:[0,T]\times \R^d\to \C^N$ we define the oscillation function at $x$ of radius $r$. For $\phi\in L^1_{\rm loc}(\R^d)$ let
\begin{align*}
\text{osc}_{r,x}(\phi) = \fint_{B_r(x)}\Big| \phi(y) - \fint_{B_r(x)} \phi(z)\ud z \Big| \ud y.
\end{align*}
Note that if  $|\phi(y)-\phi(z)|\leq \omega(\varepsilon)$ for all $y,z\in \R^d$ satisfying $|x-y|<\varepsilon$, then $\text{osc}_{r,x}\leq \omega(r)$. For discussions on this definition and the connections with VMO and BMO we refer to \cite[Chapter 6]{Kry2008}.

\begin{assumption}\label{ass:2mA} \
\begin{enumerate}[$(1)$]
\item The functions $a_{\alpha\beta}:[0,T]\times \O\times \R^d\to \C^{N\times N}$ are strongly progressively measurable.

\item There exist $\mu\in (0,1)$ and $K>0$ such that
\[\Re\Big(\sum_{|\alpha|=|\beta| = m} \xi^{\alpha} \xi^{\beta} (a_{\alpha\beta}(t,\omega,x)\theta, \theta)_{\C^N}\Big) \geq \mu |\xi|^{2m} |\theta|^2,\]
and $|a_{\alpha\beta}(t,\omega,x)|_{\C^{N\times N}}\leq K$ for all $\xi\in \R^d, \theta\in \C^N, x\in \R^d, t\in [0,T]$ and $\omega\in \O$.

\item Let $\gamma\in (0,1)$. Assume there exists an $R\in (0,\infty)$ such that for all $|\alpha|, |\beta|=m$, $r\in (0,R]$, $x\in \R^d$, $t\in [0,T]$ and $\omega\in \O$,
\[\text{osc}_{r,x}(a_{\alpha\beta}(t,\omega,\cdot))\leq \gamma.\]
\end{enumerate}
\end{assumption}
Note that in $(t,\omega)$ only measurability is assumed.

For the $\Omega$-independent setting, a slightly less restrictive condition appears in \cite[Theorem 5.2]{DKnew}. We choose the above formulation assumption in order to make the assumptions easier to state. However, it is possible to extend the results of this section to their setting.

Concerning $f$ and $g_n$ we make the following assumptions:
\begin{assumption}\label{ass:2mfg} \
\begin{enumerate}[$(1)$]
\item The function $f:[0,T]\times\O\times H^{2m,q}(\R^d;\C^N)\to L^q(\R^d;\C^N)$ is strongly
progressively measurable, $f(\cdot, \cdot,0)\in L^p(\Omega;L^p(I,w_{\alpha};H^{2m,q}(\R^d;\C^N)))$,
and there exist $L_{f}$ and $\tilde L_{f}$ such that for all $t\in [0,T]$,
$\omega\in \O$, and $u,v\in H^{2m,q}(\R^d;\C^N)$,
\begin{align*}
\|f(t,\omega, u) - &f(t,\omega,v)\|_{L^q(\R^d,\C^N)} \\ & \leq L_{f} \|D^{2m} u-D^{2m} v\|_{L^q(\R^d;\C^N)} +\tilde L_{f} \|u-v\|_{H^{2m-1,q}(\R^d;\C^N)}.
\end{align*}
\item
The functions $g_n:[0,T]\times\O\times H^{2m,q}(\R^d;\C^N) \to H^{m,q}(\R^d;\C^N)$
 are strongly progressive measurable, $(g_n)_{n\geq 1} (\cdot, \cdot, 0)\in L^p(\Omega;L^p(I,w_{\alpha};H^{m,q}(\R^d;\ell^2(\C^N)))$ and there exist $L_{g}$, $\tilde L_{g}$ such that for all $t\in
[0,T]$, $\omega\in \O$, and $u,v\in H^{2m,q}(\R^d;\C^N)$,
\begin{align*}
\|(g_n(t,\omega, u) - &g_n(t,\omega,v))_{n\geq 1}\|_{H^{m,q}(\R^d;\ell^2(\C^N))}
\\ & \leq L_{g} \|D^{2m}u-D^{2m}v\|_{L^q(\R^d;\C^N)} +
\tilde L_{g} \|u-v\|_{H^{2m-1,q}(\R^d;\C^N)}.
\end{align*}
\end{enumerate}
\end{assumption}

The nonlinearity $f$ can depend on $u, D^1 u, \ldots, D^{2m} u$ in a Lipschitz continuous way as long as the dependence on $D^{2m} u$ has a small Lipschitz constant. One could allow lower order terms in $A$, but one can just put them into the function $f$. Similarly, $g$ can depend on $u, D^1 u, \ldots, D^m u$ in a Lipschitz continuous way as long as the dependence on $D^m u$ has a small Lipschitz constant.

The main result of this section is as follows.
\begin{theorem}\label{thm:2msystemxind}
Let $T\in (0,\infty)$, $p\in (2, \infty)$, $q\in [2, \infty)$ and $\alpha\in [0,\frac{p}{2}-1)$ (or $q=r=2$ and $\alpha=0$). Set $\delta = 1- \frac{1+\alpha}{p}$ and $I=(0,T)$. Assume there exist a constant $\varrho = \varrho(p,q,\alpha,T,m,N,d,R,K,\mu)$ such that Assumptions \ref{ass:2mA} and \ref{ass:2mfg} hold with
$\gamma,L_f,L_g\in (0,\varrho)$. Then for any $u_0\in L^p(\Omega,\F_0; B^{2m\delta}_{q,p}(\R^d))$, the problem  \eqref{eq:stochsystem} has a unique strong solution $U\in L^p_{\F}(\O;L^p(I,w_{\alpha};H^{2m,q}(\R^d;\C^N)))$.
Moreover, there exist constants $C, C_{\varepsilon}, C_{\varepsilon,\theta}$ depending on $p,q,\alpha, T,m,N,d,R,K,\mu$ and the Lipschitz constants of $F$ and $G$ such that
\begin{align*}
\|U\|_{L^p(\O;C(\overline{I}; B^{2m\delta}_{q,p}(\R^d)))}& \leq C K_{u_0,f,g},
\\
\|U\|_{L^p(\O;C([\varepsilon,T]; B^{2m(1-\frac1p)}_{q,p}(\R^d)))}& \leq C_{\varepsilon} K_{u_0,f,g}, \ \ \varepsilon\in (0,T].
\\ \|U\|_{L^p(\Omega;H^{\theta,p}(I,w_{\alpha};H^{2m(1-\theta),q}(\R^d;\C^N)))} &\leq C_{\theta}K_{u_0,f,g}, \ \ \theta\in [0,\tfrac12)
\\ \|U\|_{L^p(\Omega;C^{\theta-\frac{1+\alpha}{p}}(\overline{I};H^{2m(1-\theta),q}(\R^d;\C^N)))}& \leq C_{\theta}K_{u_0,f,g}, \ \ \theta\in (\tfrac{1+\alpha}{p},\tfrac12)
\\  \|U\|_{L^p(\Omega;C^{\theta-\frac{1}{p}}([\varepsilon,T];H^{2m(1-\theta),q}(\R^d;\C^N)))}&\leq C_{\varepsilon,\theta}K_{u_0,f,g}, \ \ \theta\in (\tfrac1p,\tfrac12), \varepsilon\in (0,T].
\end{align*}
where
\begin{equation*}\label{eq:Ku0FG2m}
\begin{aligned}
K_{u_0,f,g} &= \|u_0\|_{L^p(\Omega;B^{2m\delta}_{q,p}(\R^d))}+\|f(\cdot, \cdot, 0)\|_{L^p(\Omega;L^p(I,w_{\alpha},L^q(\R^d;\C^N)))} \\ & \qquad + \|g(\cdot, \cdot, 0)\|_{L^p(\Omega;L^p(I,w_{\alpha},H^{m,q}(\R^d;\ell^2(\C^N))))}.
\end{aligned}
\end{equation*}
Furthermore, if $U^{1}, U^2$ are  the strong solution of \eqref{eq:stochsystem} with initial value $u_0^1,u_0^2\in L^p(\Omega,\F_0;B^{2m\delta}_{q,p}(\R^d))$ respectively, then each of the above estimates holds with $U$ replaced by $U^1-U^2$, and $K_{u_0,f,g}$ replaced by $K_{u_0^1-u_0^2,f,g}$ on the right-hand side.
\end{theorem}

\begin{proof}
Let $X_0 =L^q(\R^d;\C^N)$ and $X_1 =W^{2m,q}(\R^d;\C^N)$. Since the coefficients $a_{\alpha\beta}$ are uniformly bounded, Assumption \ref{conditionA} holds (with $B=0$). Let $X_\theta=H^{2m\theta,q}(\R^d;\C^N)$ for $\theta\in (0,1)$. Note that if $\theta\in [0,1]$ and $2m\theta\in \N$, then $X_\theta = W^{2m\theta,q}(\R^d;\C^N)$ (see \cite[Theorem 2.33]{Tr1}). Let $X_{\theta,p}:=(X_0, X_1)_{\theta,p} = B^{2m\theta}_{q,p}(\R^d;\C^N)$. By \cite[2.4.2 (11) and (16)]{Tr1} Assumption \ref{conditionX} holds.

On $X_0$ consider $A_0 = (1-\Delta)^{m} I_N$ with $D(A_0) =X_1$, where $I_N$ stands for the $N\times N$ diagonal operator. Then by \cite[Theorem 10.2.25]{HNVW2} and \cite[Corollary 5.5.5]{Haase:2} $0\in \rho(A_0)$ and the operator $A_0$ has a bounded $H^\infty$-calculus of angle $0$. Now from Proposition \ref{prop:suffcondSMR} we obtain that condition \eqref{it:A0cond} of Theorem \ref{thm:timedepSMR} holds. From \cite[Theorem 5.2 and Section 8]{DKnew} we deduce that the condition \eqref{it:detcond}  of Theorem \ref{thm:timedepSMR} holds. Therefore, Theorem \ref{thm:timedepSMR} shows that $A\in \SMR(p,\alpha,T)$.

It remains to check the conditions of Theorem \ref{thm:nonlinearperturbation}. From the above we see that Assumption \ref{conditionMR} holds. Furthermore we claim that for any $\varepsilon>0$, the function $F=f$ satisfies Assumption \ref{as:LipschitzF} with constants
\[L_F = L_f+\varepsilon \  \ \text{and} \ \  \tilde{L}_F = C_{\varepsilon}\tilde{L}_f.\]
Indeed, it suffices to note that for any $k\in \{1, \ldots, 2m-1\}$ (see \cite[Exercise 1.5.6]{Kry2008})
\begin{align}\label{eq:interpolationest}
\|u\|_{H^{k,q}(\R^d;\C^N)}&
 \leq \varepsilon \|D^{2m} u\|_{L^q(\R^d;\C^N)} + C_{\varepsilon,q,m} \|u\|_{L^q(\R^d;\C^N)}.
\end{align}
for all $\varepsilon>0$.

Since $\gamma(\ell^2,X_{\frac12}) = H^{m,q}(\R^d;\ell^2(\C^N))$ isomorphically (use \cite[Proposition 9.3.2]{HNVW2} and the isomorphism $A_0^{1/2}$), in a similar way as above one sees that the function $G = g$ satisfies Assumption \ref{as:LipschitzG} with $L_G = L_g+\varepsilon$ and $\tilde{L}_G = C_{\varepsilon}\tilde{L}_g$.

Now all the statements  follow from Theorem \ref{thm:nonlinearperturbation}.
\end{proof}

\begin{remark}\label{rem:higherordersmoothness}
To obtain the regularity result of Theorem \ref{thm:2msystemxind} in the whole scale of spaces $H^{s,q}(\R^d;\C^N)$ with $s\in \R$ one needs to assume smoothness on the coefficients $a_{\alpha\beta}$, and change the assumptions on $f$, $g$ and $u_0$ appropriately. Indeed, as in \cite[Section 5]{Kry} this follows by applying $(1-\Delta)^{s/2}$ to both sides of the equation and introducing $V = (1-\Delta)^{s/2} U$, where $s\in \R$. The details are left to the reader.

In \cite[Theorem 7.2]{Kry} another method to derive space-time regularity results is described. In the latter result one loses an $\varepsilon$ of regularity. The sharp identifications and embeddings obtained in \cite{NVW12a} and in the weighted case in \cite{AgrVer}, make it possible to avoid this loss in regularity.
\end{remark}

\begin{remark}\label{rem:weightsspace}
A version of Theorem \ref{thm:2msystemxind} with $A_{\frac{p}{2}}$-weights in time and $A_q$-weights in space holds as well. Here one can take any weight $v\in A_{\frac{q}{2}}$ in time and $w\in A_r$ in space.
For details we refer to \cite[Theorem 5.2 and Section 8]{DKnew} and \cite[Section 4.4]{GV}.
\end{remark}

As a consequence we obtain the following H\"older continuity result in space-time.
\begin{corollary}\label{cor:Holder2m}
Consider the setting of Theorem \ref{thm:2msystemxind} and assume $mq\geq d$ and $2m\delta-\frac{d}{q}\notin \N$. Then for every $\lambda\in (0,\delta-\frac12)$, there exists a constant $C$ such that
\[\|U\|_{L^p(\O;C^{\lambda,2m\delta-\frac{d}{q}}(\overline{I}\times\R^d;\C^N))} \leq C K_{u_0,f,g},\]
where $K_{u_0,f,g}$ is given by \eqref{eq:Ku0FG2m}
\end{corollary}
\begin{proof}
By Sobolev embedding (see \cite[Theorem 2.8.1(d) and (e)]{Tr1}) the spaces $B^{s}_{q,p}(\R^d;\C^N)$ and $H^{s,q}(\R^d;\C^N)$ continuously embed into $C^{s-\frac{d}{q}}(\R^d;\C^N)$ if $s>\frac{d}{q}$ and $s-\frac{d}{q}\notin \N$. Therefore, Theorem \ref{thm:2msystemxind} yields
\begin{align*}
\|U\|_{L^p(\O;C(\overline{I}; C^{2m\delta-\frac{d}{q}}(\R^d)))} + \|U\|_{L^p(\O;C^{\lambda}(\overline{I}; C(\R^d)))}& \leq C_{\lambda} K_{u_0,f,g},
\end{align*}
where the estimate for the second term follows by taking $\theta = \lambda+\frac{1+\alpha}{p}$ and noting that $H^{2m(1-\theta),q}(\R^d;\C^N)\hookrightarrow C(\R^d)$ which follows from
$\theta<\frac12$ and $2m(1-\theta)>\frac{d}{q}$. Combining the estimates of both terms the space-time H\"older regularity follows.
\end{proof}

\begin{remark}\label{rem:optimalHolder}
In the space variable Corollary \ref{cor:Holder2m} yields an endpoint result. Moreover, the best regularity result is obtained if $\alpha = 0$ in which case $\delta = 1-\frac1p$ which leads to the most restrictive condition on $u_0$. If Assumptions \ref{ass:2mA} and \ref{ass:2mfg} hold for arbitrary large values of $p,q\in (2, \infty)$ and $\alpha=0$, then the above result implies that for any $\varepsilon>0$
\[U\in C^{\frac12-\varepsilon,2m-\varepsilon}(\overline{I}\times\R^d;\C^N) \ \ \text{a.s.}\]
and corresponding $L^p(\Omega)$-estimates hold for any $p\in (2, \infty)$. Improved H\"older regularity results of order $s+2m$ in the space variable can be obtained if $f$ and $g$ map into $H^{s,q}$ (see Remark \ref{rem:higherordersmoothness}). It would be interesting to develop a H\"older theory for \eqref{eq:stochsystem} as it is done in \cite{DuLiuZhang} for $m=1$.
\end{remark}

\section{Parabolic systems of SPDEs of second order}\label{sec:secondorder}

In this section we discuss $L^p(L^q)$-theory for systems of second order SPDEs with rough initial values. The setting is the same as in Section \ref{sec:2morder}. However, this time we will consider $B \neq 0$. Related problems have been discussed in \cite{KimLeesystems} in an $L^p(L^p)$-setting with smooth initial values and in \cite{DuLiuZhang} in an $L^p(\Omega;L^2((0,T)\times\R^d))$-setting and a H\"older setting, with vanishing initial values.

In this section we consider the following system of stochastic PDEs on $[0,T]\times \R^d$:
\begin{equation}\label{eq:stochsystemsec}
\left\{
  \begin{array}{ll}
  dU(t)+ A(t) U(t)\ud t & = f(t,U(t))\ud t +  \sum_{n\geq 1}(b_n(t) U(t) + g_n(t,U(t)) \ud w_n(t), \\
    U(0) & = u_0,
  \end{array}
\right.
\end{equation}
where $w_n$ is a sequence of independent standard Brownian motions. The function $U:[0,T]\times \Omega\to L^p(\R^d;\C^N)$ is the unknown.

The operators $A$ and $b_n$ are given by
\begin{align*}
(A(t,\omega) \phi)(x) &= - \sum_{i,j=1}^d a_{ij}(t,\omega,x)\partial_j\partial_k \phi(x), \ \ x\in \R^d, t\in [0,T], \omega\in \Omega.
\\ (b_n(t,\omega) \phi)(x) & = \Big(\sum_{j=1}^d \sigma_{jkn}(t,\omega,x) \partial_j\phi_{k}(x)\Big)_{k=1}^N, \ \ x\in \R^d, t\in [0,T], \omega\in \Omega.
\end{align*}
There is no need to consider lower order terms in $A$ or $b_n$ since they can be absorbed in the functions $f$ and $g_n$, respectively.

\begin{remark}
The divergence form case could be treated in a similar manner. See Remark
\ref{rk:div form}.
\end{remark}

We make the following assumptions on the coefficients.
\begin{assumption}\label{ass:2ndoderpara}  \
\begin{enumerate}
\item The functions $a_{ij}:[0,T]\times \Omega\times \R^d\to \C^{N\times N}$ and $\sigma_{jkn}:[0,T]\times \Omega\times \R^d\to \R$ are strongly progressively measurable.

\item There exist $\mu\in (0,1)$ and $K>0$ such that $|a_{ij}(t,\omega,x)|_{\C^{N\times N}}\leq K$
and $\|(\sigma_{jkn}(t,\omega,\cdot))_{n\geq 1}\|_{W^{1,\infty}(\R^d;\ell^2)}\leq K$.

\[\Re\Big( \sum_{i,j=1}^d \xi_{i} \xi_{j} ((a_{ij}(t,\omega,x)-\Sigma_{ij}(t,\omega,x))\theta, \theta)_{\C^N}\Big) \geq \mu |\xi|^{2} |\theta|^2,\]
for all $\xi\in \R^d, \theta\in \C^N, x\in \R^d, t\in [0,T]$. Here for each fixed numbers $i,j\in \{1, \ldots, d\}$, $\Sigma_{ij}(t,\omega,x)$ is the $N\times N$ diagonal matrix with diagonal elements $(\frac12 \sum_{n\geq 1}\sigma_{ikn}(t,\omega,x) \sigma_{jkn}(t,\omega,x))_{k=1}^N$.

\item Assume there exists an increasing continuous function $\zeta:[0,\infty)\to [0,\infty)$ with $\zeta(0) = 0$ such that for all $i,j$, $x,y\in \R^d$, $t\in [0,T]$ and $\omega\in \O$,
\[|a_{i,j}(t,\omega,x) - a_{i,j}(t,\omega,y)| + \sum_{n\geq 1} |\sigma_{jkn}(t,\omega,x)-\sigma_{jkn}(t,\omega,y)|^2\leq \zeta(|x-y|).\]
\end{enumerate}
\end{assumption}

We start with the $x$-independent case.
\begin{theorem}\label{thm:secondorderxind}
Let $T\in (0,\infty)$, $p\in (2, \infty)$, $q\in [2, \infty)$ and $\alpha\in [0,\frac{p}{2}-1)$ (or $p=q=2$ and $\alpha=0$). Set $\delta = 1- \frac{1+\alpha}{p}$ and $I=(0,T)$.
Suppose Assumption \ref{ass:2ndoderpara} holds, Assumption \ref{ass:2mfg} holds with $m=1$, and suppose further that $a_{ij}$ and $\sigma_{jkn}$ are $x$-independent. Then for any $u_0\in L^p(\Omega,\F_0; B^{2\delta}_{q,p}(\R^d))$, the problem  \eqref{eq:stochsystemsec} has a unique strong solution $U\in L^p_{\F}(\O;L^p(I,w_{\alpha};H^{2,q}(\R^d;\C^N)))$.
Moreover, there exist constants $C, C_{\varepsilon}, C_{\varepsilon,\theta}$ depending on $p,q,\alpha, T,N,d,K,\mu$ and the Lipschitz constants of $f$ and $g$ such that
\begin{align*}
\|U\|_{L^p(\O;C(\overline{I}; B^{2\delta}_{q,p}(\R^d)))}& \leq C K_{u_0,f,g},
\\
\|U\|_{L^p(\O;C([\varepsilon,T]; B^{2(1-\frac1p)}_{q,p}(\R^d)))}& \leq C_{\varepsilon} K_{u_0,f,g}, \ \ \varepsilon\in (0,T].
\\ \|U\|_{L^p(\Omega;H^{\theta,p}(I,w_{\alpha};H^{2(1-\theta),q}(\R^d;\C^N)))} &\leq C_{\theta}K_{u_0,f,g}, \ \ \theta\in [0,\tfrac12)
\\ \|U\|_{L^p(\Omega;C^{\theta-\frac{1+\alpha}{p}}(\overline{I};H^{2(1-\theta),q}(\R^d;\C^N)))}& \leq C_{\theta}K_{u_0,f,g}, \ \ \theta\in (\tfrac{1+\alpha}{p},\tfrac12)
\\  \|U\|_{L^p(\Omega;C^{\theta-\frac{1}{p}}([\varepsilon,T];H^{2(1-\theta),q}(\R^d;\C^N)))}&\leq C_{\varepsilon,\theta}K_{u_0,f,g}, \ \ \theta\in (\tfrac1p,\tfrac12), \varepsilon\in (0,T].
\end{align*}
where
\begin{equation*}
\begin{aligned}
K_{u_0,f,g} &= \|u_0\|_{L^p(\Omega;B^{2\delta}_{q,p}(\R^d))}+\|f(\cdot, \cdot, 0)\|_{L^p(\Omega;L^p(I,w_{\alpha},L^q(\R^d;\C^N)))} \\ & \qquad + \|(g_{n}(\cdot, \cdot, 0))_{n\in \N}\|_{L^p(\Omega;L^p(I,w_{\alpha},H^{1,q}(\R^d;\ell_{2})))}.
\end{aligned}
\end{equation*}
Furthermore, if $U^{1}, U^2$ are  the strong solution of \eqref{eq:stochsystemsec} with initial value $u_0^1,u_0^2\in L^p(\Omega,\F_0;B^{2\delta}_{q,p}(\R^d))$ respectively, then each of the above estimates holds with $U$ replaced by $U^1-U^2$, and $K_{u_0,f,g}$ replaced by $K_{u_0^1-u_0^2,f,g}$ on the right-hand side.
 \end{theorem}
\begin{proof}
Define the function spaces $X_{\theta} = H^{2\theta,q}$ and $X_{\theta, p}=(X_{0},X_{1})_{\theta,p}$ as in Theorem \ref{thm:2msystemxind} with $m=1$. Clearly, Assumptions \ref{conditionX} and \ref{conditionA} hold. In order to prove the result we will check the conditions of Theorem \ref{thm:nonlinearperturbation} again. The proof is the same as for Theorem \ref{thm:2msystemxind}, but this time we need to check $(A,B)\in \SMR(p,\alpha, T)$ with a nonzero $B$. Indeed, let
$B(t)u  = \sum_{k=1}^N \sum_{j=1}^d b_{jk}(t) B_{jk}$ with
\[b_{jk}(t)h = \sum_{n=1}^\infty \sigma_{jkn}(t) h_n,  \ \ \text{and} \ \ (B_{jk}\phi)_{\ell} = \delta_{k\ell}\partial_j\phi_{k},\]
where $\delta_{k\ell}$ is the Kronecker symbol.
Then $B_{jk}$ generates a translation group on $L^q(\R^d;\C^{N})$ given by $(e^{t B_{jk}} u(x))_\ell =u_\ell(x+\delta_{k\ell} t e_j)$, where $e_j$ denotes the $j$-th unit vector in $\R^d$. Then Assumption \ref{conditionB} is fulfilled, thanks to the fact that the coefficients are $x$-independent. Moreover,
\[([B(t),B(t)]u)_{\ell} = \sum_{i,j=1}^d \sum_{n\geq 1} \sigma_{i\ell n}(t) \sigma_{j\ell n}(t) \partial_i\partial_j u_{\ell}.\]
Letting $\wt{A}(t) = A(t) +\frac12[B(t),B(t)]$, Theorem \ref{thm:redB0} gives that $(A,B)\in \SMR(p,\alpha, T)$ if and only if $(\wt{A},0)\in \SMR(p, \alpha, T)$. However, by Assumption  \ref{ass:2ndoderpara} the operator $\wt{A}$ fulfils Assumption \ref{ass:2mA} (with $m=1$), and therefore, as in the proof of Theorem \ref{thm:2msystemxind}, we find that  $(\wt{A},0)\in \SMR(p, \alpha, T)$.
\end{proof}

In the $x$-dependent case we obtain the following, where unlike in Theorems \ref{thm:2msystemxind} and \ref{thm:secondorderxind} we have to take $p=q$.
\begin{theorem}\label{thm:secondorderxdep}
Let $T\in (0,\infty)$, $p\in [2, \infty)$ and $\alpha\in [0,\frac{p}{2}-1)$ (where $\alpha=0$ if $p=2$). Set $\delta = 1- \frac{1+\alpha}{p}$ and $I=(0,T)$.
Suppose Assumption \ref{ass:2ndoderpara} holds and \ref{ass:2mfg} holds with $m=1$. Then for any $u_0\in L^p(\Omega,\F_0; B^{2\delta}_{p,p}(\R^d))$, the problem  \eqref{eq:stochsystemsec} has a unique strong solution $U\in L^p_{\F}(\O;L^p(I,w_{\alpha};H^{2,p}(\R^d;\C^N)))$.
Moreover, there exist constants $C, C_{\varepsilon}, C_{\varepsilon,\theta}$ depending on $p,\alpha, T,N,d,\zeta,K,\mu$ and the Lipschitz constants of $f$ and $g$ such that
\begin{align*}
\|U\|_{L^p(\O;C(\overline{I}; B^{2\delta}_{p,p}(\R^d)))}& \leq C K_{u_0,f,g},
\\
\|U\|_{L^p(\O;C([\varepsilon,T]; B^{2(1-\frac1p)}_{p,p}(\R^d)))}& \leq C_{\varepsilon} K_{u_0,f,g}, \ \ \varepsilon\in (0,T].
\\ \|U\|_{L^p(\Omega;H^{\theta,p}(I,w_{\alpha};H^{2(1-\theta),p}(\R^d;\C^N)))} &\leq C_{\theta}K_{u_0,f,g}, \ \ \theta\in [0,\tfrac12)
\\ \|U\|_{L^p(\Omega;C^{\theta-\frac{1+\alpha}{p}}(\overline{I};H^{2(1-\theta),p}(\R^d;\C^N)))}& \leq C_{\theta}K_{u_0,f,g}, \ \ \theta\in (\tfrac{1+\alpha}{p},\tfrac12)
\\  \|U\|_{L^p(\Omega;C^{\theta-\frac{1}{p}}([\varepsilon,T];H^{2(1-\theta),p}(\R^d;\C^N)))}&\leq C_{\varepsilon,\theta}K_{u_0,f,g}, \ \ \theta\in (\tfrac1p,\tfrac12), \varepsilon\in (0,T].
\end{align*}
where
\begin{equation*}
\begin{aligned}
K_{u_0,f,g} &= \|u_0\|_{L^p(\Omega;B^{2\delta}_{p,p}(\R^d))}+\|f(\cdot, \cdot, 0)\|_{L^p(\Omega;L^p(I,w_{\alpha},L^p(\R^d;\C^N)))} \\ & \qquad + \|(g_{n}(\cdot, \cdot, 0))_{n\in \N}\|_{L^p(\Omega;L^p(I,w_{\alpha},H^{1,p}(\R^d;\ell_{2})))}.
\end{aligned}
\end{equation*}
Furthermore, if $U^{1}, U^2$ are  the strong solution of \eqref{eq:stochsystemsec} with initial value $u_0^1,u_0^2\in L^p(\Omega,\F_0;B^{2\delta}_{p,p}(\R^d))$ respectively, then each of the above estimates holds with $U$ replaced by $U^1-U^2$, and $K_{u_0,f,g}$ replaced by $K_{u_0^1-u_0^2,f,g}$ on the right-hand side.
\end{theorem}

To prove this, arguing as in Theorem \ref{thm:secondorderxind} it follows from Theorem \ref{thm:nonlinearperturbation} that it suffices to prove that $(A,B)\in \SMR(p,\alpha,T)$. We will use a variation of a localization argument of \cite[Section 6]{Kry}. For this we start with the following lemma.

\begin{lemma}[Freezing lemma]\label{lem:Freeze}
Let $T\in (0,\infty)$, $p\in [2, \infty)$ and $\alpha\in [0,\frac{p}{2}-1)$ (where $\alpha=0$ if $p=2$). Set $\delta = 1- \frac{1+\alpha}{p}$ and $I=(0,T)$.
Suppose Assumption \ref{ass:2ndoderpara} holds. Let $f\in L^p_{\F}(\Omega\times I,w_{\alpha};L^p(\R^d;\C^N))$ and $g\in L^p_{\F}(\Omega\times I,w_{\alpha};H^{1,p}(\R^d;\ell^2))$.
Assume $U$ is a strong solution of
\begin{equation}\label{eq:stochsystemsechulp}
\left\{
  \begin{array}{ll}
  dU(t)+ A(t) U(t)\ud t & = f(t)\ud t +  \sum_{n\geq 1}(b_n(t) U(t) + g_n(t)) \ud w_n(t), \\
    U(0) & = 0,
  \end{array}
\right.
\end{equation}
There exists an $\varepsilon =  \varepsilon(p,q,\alpha, T,N,d,K,\mu)$ such that
if $U$ has support in $B_{\varepsilon}(x_0) = \{x\in \R^d:|x-x_0|<\varepsilon\}$ for some $x_0$, then for each $\theta\in [0,\frac12)$ there is a constant $C$ such that
\begin{equation}\label{eq:MRestfreeze}
\begin{aligned}
&\|U\|_{L^p(\Omega;H^{\theta,p}(I,w_{\alpha};H^{2(1-\theta),p}(\R^d)))} \\ &\qquad \leq
C\|f\|_{L^{p}(\Omega; L^p(I,w_{\alpha};L^p(\R^d;\C^N)))} + C\|(g_{n})_{n\in\N}\|_{L^{p}(\Omega; L^p(I,w_{\alpha};H^{1,p}(\R^d;\ell^2)))}.
\end{aligned}
\end{equation}
\end{lemma}
\begin{proof}
Without loss of generality we can assume $x_0 = 0$. In order to simplify the notation let
\begin{align*}
Y_{\theta,\eta,t} &= L^p(\Omega;H^{\theta,p}((0,t),w_{\alpha};H^{2\eta,p}(\R^d;\C^N)))
\\  Y_{\theta,\eta,t}(\ell^2) &= L^p(\Omega;H^{\theta,p}((0,t),w_{\alpha};H^{2(1-\eta),p}(\R^d;\ell^2(\C^N)))).
\end{align*}

Let $\tilde{A}(t)$ and $\tilde{b_n}(t)$ be given by
\begin{align*}
\tilde{A}(t)\phi& = - \sum_{i,j=1}^d a_{ij}(t,\omega,0)\partial_j \partial_k \phi(x),
\\ (\tilde{b}_n(t,\omega) \phi)(x) & = \Big(\sum_{j=1}^d \sigma_{jkn}(t,\omega,0) \partial_j \phi_{k}(x)\Big)_{k=1}^N.
\end{align*}
Furthermore, let
\[\tilde{f}(\cdot,U):=f(\cdot,U) + (\tilde{A} - A)U, \ \ \text{and} \ \ \tilde{g}_n(\cdot,U):= g_n(\cdot,U)+(\tilde{b}_n-b_n)U.\]
Clearly, $U$ satisfies
\begin{align*}
dU(t)+ \tilde{A}(t) U(t)\ud t  & = \tilde{f}(t,U(t)))\ud t +  \sum_{n\geq 1}(\tilde{b}_n(t) U(t) +\tilde{g}_n(t,U(t))) \ud w_n(t).
\end{align*}
Therefore, by Theorem \ref{thm:secondorderxind}
\begin{align*}
\|U\|_{Y_{\theta,(1-\theta),T}} &\leq C_{\theta}K_{0,\tilde{f},\tilde{g}}
\\ & \leq  C_{\theta}K_{0,f,g} + C\|(\tilde{A} - A)U\|_{Y_{0,0,T}}  + C \|((\tilde{b}_n-b_n)U)_{n\geq 1}\|_{Y_{0,1,T}}.
\end{align*}
To estimate the latter note that by Assumption \ref{ass:2ndoderpara} and the support condition on $U$, we have
\begin{align*}
\|(\tilde{A} - A)U\|_{Y_{0,0,T}}\leq \zeta(\varepsilon) \|U\|_{Y_{0,1,T}}.
\end{align*}
Similarly, for the $b_n$-term, by the product rule and Assumption \ref{ass:2ndoderpara}, we obtain (with $K$ as in Assumption \ref{ass:2ndoderpara}) that for all $t \in [0,T]$,
\begin{align*}
\|((\tilde{b}_n-b_n)U(t))_{n\geq 1}\|_{W^{1,p}(\R^d;\ell^2(\C^N)))}& \leq K \|U(t)\|_{W^{1,p}(\R^d;\C^N)} + \zeta(\varepsilon) \|U(t)\|_{W^{2,p}(\R^d;\C^N)}
\\ & \leq C_{\varepsilon} \|U(t)\|_{L^p(\R^d;\C^N)} + (\zeta(\varepsilon) +\varepsilon)\|U(t)\|_{W^{2,p}(\R^d;\C^N)},
\end{align*}
where in the last step we used \cite[Corollary 1.5.2]{Kry2008}. We can conclude that
\begin{align}\label{eq:niceest1}
\|U\|_{Y_{\theta,(1-\theta),T}} & \leq C_{\theta}K_{0,f,g} + C (2\zeta(\varepsilon) + \varepsilon)\|U\|_{Y_{0,1,T}} + C_{\varepsilon} \|U\|_{Y_{0,0,T}}.
\end{align}
Now let $\theta = 0$ and choose $\varepsilon>0$ such that $C (2\zeta(\varepsilon) + \varepsilon)\leq \frac12$. Then we obtain
\begin{align}\label{eq:niceest2}
\|U\|_{Y_{0,1,T}} \leq CK_{0,f,g} + C\|U\|_{Y_{0,0,T}}.
\end{align}
The same estimate holds with $T$ replaced by $t$.

Since $U$ is a strong solution of \eqref{eq:stochsystemsechulp}, Lemma \ref{lem:simpleestt}, the properties of $A$ and $b_n$, and \eqref{eq:niceest2}
give that, for all $t\in [0,T]$,
\begin{align*}
\|U(t)\|_{L^p(\Omega;L^p(\R^d;\C^N))} &\leq C\|f\|_{Y_{0,0,t}} + C\|g\|_{Y_{0,\frac12,t}(\ell^2)} + C\|U\|_{Y_{0,1,t}}
\\ & \leq C\|f\|_{Y_{0,0,t}} + C\|g\|_{Y_{0,\frac12,t}(\ell^2)} + C\|U\|_{Y_{0,0,t}}.
\end{align*}
Therefore, Gronwall's lemma gives that for all $t\in [0,T]$,
\[\|U(t)\|_{L^p(\Omega;L^p(\R^d;\C^N))} \leq C\|f\|_{Y_{0,0,T}} + C\|g\|_{Y_{0,\frac12,T}(\ell^2)}.\]
and thus
$$\|U\|_{Y_{0,0,T}} \leq 2CK_{0,f,g}.$$
Substituting the latter estimate in \eqref{eq:niceest1} and \eqref{eq:niceest2} we find that for all $\theta\in [0,\frac12)$,
\[\|U\|_{Y_{\theta,(1-\theta),T}}  \leq C_{\theta}K_{0,f,g}.\]
\end{proof}

To be able to apply the freezing lemma, one also needs the following elementary Fubini result, which is trivial if $I = \R$, but for bounded intervals can still be deduced from an interpolation argument.

\begin{lemma}
\label{lem:loc}
Let $\theta \in [0,1]$, $p \in (1,\infty)$ and $w\in A_p$. Let $I= (0,T)$ for some $T\in (0,\infty]$.
Let $\phi \in L^{p}(\R^{d})$ be such that $\|\phi\|_{L^{p}(\R^{d})} = 1$.
We have that
$$
\|f\|_{H^{\theta,p}(I,w;L^{p}(\R^{d};\C^N))} \eqsim \Big( \int \limits _{\R^{d}}
\|(t,x) \mapsto \phi(x-\xi)f(t,x)\|_{H^{\theta,p}(I,w;L^{p}(\R^{d};\C^N))} ^{p} \ud\xi \big)^{\frac{1}{p}},
$$
for all $f \in H^{\theta,p}(I,w;L^{p}(\R^{d}))$.
\end{lemma}

\begin{proof}
Let $\psi \in L^{p'}(\R^{d})$ of norm one be such that $\lb \phi, \psi\rb =1$. Consider the operators defined by
\begin{align*}
Pf(t,x,\xi) &= \phi(x-\xi)f(t,x) \quad t \in I, \quad x,\xi \in\R^{d}, \\
QF(t,x) & = \int \limits _{\R^{d}} \psi(x-\xi)F(t,x,\xi) \ud \xi \quad t \in I, \quad x \in\R^{d},
\end{align*}
for $f \in L^p(I,w;L^{p}(\R^{d};\C^N))$ and $F \in L^p(\R^d;L^p(I,w;L^{p}(\R^{d};\C^N)))$. Note that $QPf = f$.
By complex interpolation (see Proposition \ref{prop:funcspacesUMD}), it is thus enough to show that
\begin{align*}
P:W^{\theta,p}(I,w;L^{p}(\R^{d};\C^N))\to L^{p}(\R^{d};W^{\theta,p}(I,w;L^{p}(\R^{d};\C^N))), \\
Q:L^{p}(\R^{d};W^{\theta,p}(I,w;L^{p}(\R^{d};\C^N)))\to W^{\theta,p}(I,w;L^{p}(\R^{d};\C^N)),
\end{align*}
for $\theta = 0,1$. Let us consider $\theta = 0$ first.
For $Q$, by H\"older's inequality we have that
\begin{align*}
\|QF\|_{L^{p}(I,w;L^{p}(\R^{d};\C^N)))}
& = \Big( \int \limits _{I\times \R^{d}} |\int \limits _{\R^{d}} \psi(x-\xi)F(t,x,\xi)\ud \xi|^{p}  w(t) \ud t \ud x \Big)^{\frac{1}{p}}
\\
& \leq  \|F\|_{L^p(\R^d;L^{p}(I,w;L^{p}(\R^{d};\C^N)))}.
\end{align*}
and hence $\|Q\|\leq 1$ for $\theta=0$. The above inequalities with $f$ replaced by $\partial_{t}f$, and $F$ replaced by $\partial _{t} F$, then gives $\|Q\|\leq 1$ for $\theta=1$. By Fubini it is straightforward to check that $P$ is an isometry for $\theta = 0,1$, and thus the result follows.
\end{proof}

We also need the following simple commutator formula.
\begin{lemma}\label{lem:commutatorRiesz}
Let $\phi\in C^1_c(\R^d)$ and $\psi\in W^{1,p}(\R^d)$. Then for all $s\in (0,1)$
\begin{align*}
(-\Delta)^{s/2}(\phi \psi)(x) - \phi(x) (-\Delta)^{s/2}(\psi)(x)  =
c_{d,s}\int_{\R^d} (\phi(x+h) - \phi(x)) |h|^{-s-d} \psi(x+h)   \ud h
\end{align*}
\end{lemma}
\begin{proof}
The identity immediately follows from the following well-known identity (see \cite[Theorem 1.1(e)]{Kwa17})
\[(-\Delta)^{s/2} \psi(x) = c_{d,s} \int_{\R^d} \frac{\psi(x+h) - \psi(x)}{|h|^{d+s}}\ud h.\]
\end{proof}

\begin{proof}[Proof of Theorem \ref{thm:secondorderxdep}]
As already mentioned it suffices to prove that $(A,B)\in \SMR(p,\alpha,T)$
(in particular, one only has to treat the problem with $u_{0}=0$).
To prove this we will use Proposition \ref{prop:methodcont} and Lemma \ref{lem:Freeze}.

{\em Step 1:}  Let $\theta \in (0,\frac{1}{2})$ and $U \in Y_{0,1,T} \cap Y_{\theta,1-\theta,T}$ be a solution to \eqref{eq:stochsystemsec}. We first prove an a priori estimate in the $Y_{0,1,T}$ norm. Let $\varepsilon>0$ be as in Lemma \ref{lem:Freeze}. Let $\phi\in C^\infty_c(\R^d)$ be such that $\supp(\phi)\subseteq B_{\varepsilon}(0)$ and $\|\phi\|_{L^p(\R^d)} = 1$. Fix $\xi\in\R^d$ and let $V^{\xi}(t)(x) = U(t)(x) \phi(x-\xi)$. Then $\supp V^{\xi}\subseteq [0,T]\times B_{\varepsilon}(\xi)$, $V^{\xi}(t)(0) = 0$ and $V^{\xi}$ satisfies
\begin{align*}
 dV^{\xi}(t)+ A(t) V^{\xi}(t)\ud t & = \tilde{f}(t)\ud t +  \sum_{n\geq 1}(b_n(t) V^{\xi}(t) + \tilde{g}_n(t)) \ud w_n(t),
\end{align*}
where
\begin{align*}
\tilde{f} & = \phi(\cdot-\xi) f  + A V^{\xi} - \phi(\cdot-\xi) A U = \phi(\cdot-\xi)f + \sum_{|\beta|\leq 2, |\gamma|\leq 1} c_{\beta,\gamma} (\partial^{\beta}\phi(\cdot-\xi)) \partial^{\gamma} U,
\\ \tilde{g}_n & = \phi(\cdot-\xi) g_n  + \phi(\cdot-\xi) b_n U - b_n V^{\xi} = \phi(\cdot-\xi)g_n + \Big(\sum_{j=1}^d \sigma_{jkn}(t,\omega,x) \partial_j\phi(x-\xi) U_k\Big)_{k=1}^N
\end{align*}
for some coefficients $c_{\beta,\gamma}\in L^\infty((0,T)\times\Omega\times \R^d;\C^{N\times N})$. Now Lemma \ref{lem:Freeze} (in the notation of its proof) implies that for any $\theta'\in [0,\frac12)$ there exists a $C$ such that
\begin{equation}\label{eq:thetaVxiest}
\begin{aligned}
\|V^{\xi}\|_{Y_{\theta', 1-\theta', T}} &\leq C\|\tilde{f}\|_{Y_{0, 0, T}} + C\|\tilde{g}\|_{Y_{0,\frac12,T}(\ell^2)}
\\ & \leq C\|\phi(\cdot-\xi) f\|_{Y_{0, 0, T}} + C\|\phi(\cdot-\xi) g\|_{Y_{0,\frac12,T}(\ell^2)} \\ & \qquad +
C\sum_{|\beta|\leq 2} \|\partial^{\beta}(\phi(\cdot-\xi)) U\|_{Y_{0, \frac12, T}},
\end{aligned}
\end{equation}
First let $\theta'=0$. For $|\alpha|=2$ we have (by the product rule)
\begin{align*}
\|\phi(\cdot-\xi)\partial^{\alpha}U(t)\|_{L^p(\R^d)} \leq \|V^{\xi}(t)\|_{W^{2,p}(\R^d)} + C\sum_{|\beta|\leq 2, |\gamma|\leq 1}\|\partial^{\beta} \phi(\cdot-\xi) \partial^{\gamma} U(t)\|_{L^p(\R^d)}.
\end{align*}
Therefore, after an integration over $\xi\in \R^d$ using Lemma \ref{lem:loc}, and using the classical identity $W^{k,p}(\R^d;\C^N) = H^{k,p}(\R^d;\C^N)$ for $k\in \N$ we can conclude
\begin{align*}
\|U\|_{Y_{0, 1, T}} & \leq C\|f\|_{Y_{0, 0, T}} + C\|g\|_{Y_{0,\frac12,T}(\ell^2)} +
C\|U\|_{Y_{0, \frac12, T}}.
\end{align*}
By using interpolation inequality \eqref{eq:interpolationest} for $k=m=1$ we find
\begin{align*}
\|U\|_{Y_{0, 1, T}} & \leq C\|f\|_{Y_{0, 0, T}} + C\|g\|_{Y_{0,\frac12,T}(\ell^2)} +
C\|U\|_{Y_{0, 0, T}}.
\end{align*}
The same is true with $T$ replaced by $t$. Therefore, the term $\|U\|_{Y_{0, 0, T}}$ can be estimated by Gronwall's lemma in the same way as in the proof of Lemma \ref{lem:Freeze}. Therefore, we obtain
\begin{align}\label{eq:aprioritheta0}
\|U\|_{Y_{0, 1, T}} & \leq C\|f\|_{Y_{0, 0, T}} + C\|g\|_{Y_{0,\frac12,T}(\ell^2)}.
\end{align}
This is the required a priori estimate.

\medskip

{\em Step 2:}  Let $\theta \in (0,\frac{1}{2})$ and $U \in Y_{0,1,T} \cap Y_{\theta,1-\theta,T}$ be a solution to \eqref{eq:stochsystemsec}.  We now prove an a priori estimate in the $Y_{\theta,1-\theta,T}$ norm.
For this we use \eqref{eq:thetaVxiest} again. Observe that
\begin{align*}
\phi(\cdot-\xi) \partial_j U =  \partial_j V^{\xi} + \partial_j(\phi(\cdot-\xi)) U .
\end{align*}
Therefore, by Lemma \ref{lem:commutatorRiesz} we can write
\begin{align*}
\phi(\cdot-\xi) & (-\Delta)^{\frac12-\theta} \partial_j U =  (-\Delta)^{\frac12-\theta}(\phi(\cdot-\xi)\partial_j U)
\\ & -C_{d,\theta}\int_{\R^d} (\phi(\cdot+h-\xi) - \phi(\cdot-\xi)) |h|^{-1+2\theta-d} \partial_j U(\cdot+h)   \ud h
 \\ & = (-\Delta)^{\frac12-\theta} \partial_j V^{\xi} + (-\Delta)^{\frac12-\theta} ( \partial_j(\phi(\cdot-\xi)) U)
 \\ & \qquad - C_{d,\theta}\int_{\R^d} (\phi(\cdot+h-\xi) - \phi(\cdot-\xi)) |h|^{-1+2\theta-d} \partial_j U(\cdot+h)   \ud h
\end{align*}
It thus follows from \eqref{eq:thetaVxiest} that
\begin{align*}
\|\phi(\cdot-\xi)  &(-\Delta)^{\frac12-\theta} \partial_j U\|_{Y_{\theta,0,T}}
  \leq C\|\phi(\cdot-\xi) f\|_{Y_{0, 0, T}} + C\|\phi(\cdot-\xi) g\|_{Y_{0,\frac12,T}(\ell^2)} \\ & \qquad \qquad\qquad \qquad\qquad \qquad+
C\sum_{|\beta|\leq 2} \|\partial^{\beta}(\phi(\cdot-\xi)) U\|_{Y_{0, \frac12, T}}
\\ & \qquad \qquad \qquad\qquad \qquad\qquad+C\|(-\Delta)^{\frac12-\theta} ( \partial_j(\phi(\cdot-\xi))  U)\|_{Y_{\theta, 0, T}}
 \\ & +C\int_{\R^d} \Big\|(\phi(\cdot+h-\xi) - \phi(\cdot-\xi)) |h|^{-1+2\theta-d}   \partial_j U(\cdot)(\cdot+h) \Big\|_{Y_{\theta,0,T}}  \ud h
\end{align*}
By Lemma \ref{lem:equivalentnormsH} and Proposition \ref{prop:funcspacesUMD} we have
\begin{align*}
\|(-\Delta)^{\frac12-\theta} ( \partial_j(\phi(\cdot-\xi)) U)\|_{Y_{\theta, 0, T}}
& \leq C \| \partial_j(\phi(\cdot-\xi)) U\|_{Y_{\theta, \frac12-\theta, T}}
\\ & \leq C\| \partial_j(\phi(\cdot-\xi)) U\|_{Y_{\theta, \frac12, T}}
+ C \|( \partial_j(\phi(\cdot-\xi)) U\|_{Y_{\theta, 0, T}}
\end{align*}
Therefore, substituting this in the penultimate estimate and using Lemma \ref{lem:loc}, we obtain (noting that $1-\theta>\frac{1}{2}$ so $U\in Y_{\theta, \frac{1}{2},T}$):
\begin{equation}\label{eq:inbetweenest}
\begin{aligned}
\|(-\Delta)^{\frac12-\theta} \partial_j U\|_{Y_{\theta,0,T}}
 & \leq C\|f\|_{Y_{0, 0, T}} + C\|g\|_{Y_{0,\frac12,T}(\ell^2)}+ C \|U\|_{Y_{\theta, \frac12, T}}  + R.
\end{aligned}
\end{equation}
where $R$ is given by
\begin{align*}
R:= \int_{\R^d} \Big\|(\phi(\cdot+h-\xi) - \phi(\cdot-\xi)) |h|^{-1+2\theta-d}  \partial_jU(\cdot)(\cdot+h) \Big\|_{L^p(\R^d,\ud \xi;Y_{\theta,0,T})}  \ud h.
\end{align*}
Now write $R = R_1+ R_2$, where $R_1$ is the part of the integral for $|h|<1$ and $R_2$ is the part of the integral for $|h|\geq 1$.
Using the fundamental theorem of calculus, Fubini, the triangle inequality and the translation invariance of $Y_{\theta,0,T}$ in the space variable, we obtain
\begin{align*}
R_1& \leq \int_0^1 \int_{|h|<1} \Big\| |\nabla \phi(\cdot+ s h-\xi)| |h|^{2\theta-d}  \partial_jU(\cdot)(\cdot+h) \Big\|_{L^p(\R^d,\ud \xi;Y_{\theta,0,T})} \ud h \ud s
\\ & = C \| \partial_jU\|_{Y_{\theta,0,T}},
\end{align*}
where $C = C_{\phi} \int_{|h|<1}|h|^{2\theta-d} \ud h$. Similarly,
\begin{align*}
R_2& \leq \int_{|h|\geq 1 }\Big\|(\phi(\cdot+h-\xi) - \phi(\cdot-\xi)) |h|^{-1+2\theta-d} \partial_jU(\cdot)(\cdot+h) \Big\|_{L^p(\R^d,\ud \xi;Y_{\theta,0,T})}  \ud h
\\ & = C\|\partial_jU\|_{Y_{\theta,0,T}},
\end{align*}
where $C= C_{\phi} \int_{|h|\geq1}  |h|^{-1+2\theta-d} \ud h$. Combining this with \eqref{eq:inbetweenest}, summing over $j$ and using Lemma  \ref{lem:equivalentnormsH2} we obtain
\begin{align*}
\|U\|_{Y_{\theta,1-\theta,T}}
 & \leq C\|f\|_{Y_{0, 0, T}} + C\|g\|_{Y_{0,\frac12,T}(\ell^2)}+
 C\|U\|_{Y_{\theta, \frac12, T}}
\end{align*}
Now using the interpolation estimate of Proposition \ref{prop:funcspacesUMD} once more (which is possible because $1-\theta>\frac{1}{2}$), we can conclude that
\begin{align*}
\|U\|_{Y_{\theta,1-\theta,T}}
 & \leq C\|f\|_{Y_{0, 0, T}} + C\|g\|_{Y_{0,\frac12,T}(\ell^2)}+C \|U\|_{Y_{\theta, 0, T}}.
\end{align*}

In order to estimate $\|U\|_{Y_{\theta, 0, T}}$ note that by Lemma \ref{lem:simpleestt}, \eqref{eq:realcomplexconnect} and \eqref{eq:aprioritheta0},
\begin{align*}
\|U\|_{Y_{\theta, 0, T}} &\leq  C\|f\|_{Y_{0, 0, T}} + C\|g\|_{Y_{0,\frac12,T}(\ell^2)}+C \|A U\|_{Y_{0, 0, T}}
\\ & \leq C\|f\|_{Y_{0, 0, T}} + C\|g\|_{Y_{0,\frac12,T}(\ell^2)}+C \|U\|_{Y_{0, 1, T}}
\\ & \leq C\|f\|_{Y_{0, 0, T}} + C\|g\|_{Y_{0,\frac12,T}(\ell^2)}.
\end{align*}
We thus have the a priori estimate:
$$
\|U\|_{Y_{\theta,1-\theta,T}} \lesssim \|f\|_{Y_{0, 0, T}} + \|g\|_{Y_{0,\frac12,T}(\ell^2)}.
$$

{\em Step 3:} Now to prove the existence of a solution in $Y_{0,1,T} \cap Y_{\theta,1-\theta,T}$ let $\tilde{A} = -\Delta$. Then by Proposition \ref{prop:suffcondSMR} the problem \eqref{eq:stochsystemsec} with $A$ replaced by $\tilde{A}$ has a unique strong solution in $Y_{0,1,T} \cap Y_{\theta,1-\theta,T}$. Moreover, letting $A_{\lambda}$ and $B_{\lambda}$ be as in Proposition \ref{prop:methodcont} it follows from the previous steps that the a priori bounds hold uniformly in $\lambda\in [0,1]$. Therefore, Proposition \ref{prop:methodcont} and the text below it gives the existence of a solution $U \in Y_{0,1,T} \cap Y_{\theta,1-\theta,T}$ to \eqref{eq:stochsystemsec}.
\end{proof}

\begin{remark}
To obtain regularity in the scale $H^{s,p}(\R^d;\C^N)$ for $s\in \R$ in Theorem \ref{thm:secondorderxdep}, Remark \ref{rem:higherordersmoothness} applies again.
\end{remark}

\begin{remark}\
\begin{enumerate}
\item It would be natural to ask for an $L^p(L^q)$-theory in Theorem \ref{thm:secondorderxdep}. In \cite{GV} $A_p$-weights in time, in combination with Rubio de Francia extrapolation techniques, have been used to derive the case $p\neq q$ from $p=q$, in the case of continuous coefficients in time. This was later extended to VMO coefficients in space in \cite{DKnew}. The extrapolation technique would be applicable here as well, but it only gives regularity in $L^p(0,T,L^q(\Omega;L^q(\R^d;\C^N)))$ and not in $L^p((0,T)\times\Omega;L^q(\R^d;\C^N)))$ as one would like.
\item Another natural question is whether Theorem \ref{thm:secondorderxdep} holds if the coefficients are only VMO in the space variable. Some results in this direction have been found for equations in divergence form in \cite{Kry09VMO}.
\end{enumerate}
\end{remark}

\begin{remark}
Let us motivate that in the commuting case the assumption that the operators $b_n$ as defined below \eqref{eq:stochsystemsec} is not far from the general case. Indeed, assume that $B_1, \ldots, B_J$, are differential operators of order one which generate commuting groups on $L^q(\R^d;\C^{N\times N})$ for all $q\in (1, \infty)$. Then  we can write $B_j u =  \sum_{k=1}^d M_{jk} \partial_k u  + N_j u$, where $M_{jk}, N_j$ are $N\times N$ matrices. Then the $M_{jk}$ have real eigenvalues since otherwise the Fourier symbol would be unbounded (use \cite[Theorem 2.5.16]{Grafakos1} to reduce to one fixed direction $k$). Now by \cite[Theorem 0.1]{Brenner73} the matrices $(M_{jk})_{k=1}^d$ commute and are diagonalizable.
Moreover, if the groups $(e^{\cdot B_j})_{j=1}^J$ are commuting, then the operators $(M_{jk})_{j=1,k=1}^{J,d}$ are commuting as well. Therefore, a standard result from linear algebra implies that $(M_{jk})_{j=1,k=1}^{J,d}$ are simultaneously diagonalizable. Hence by a coordinate transformation we could have assumed that $M_{jk}$ are diagonal matrices with real entries.

Of course to reduce to this setting in a general set-up the coordinate transformations would become $(t,\omega)$-dependent and then the reduction breaks down. Even in the $(t,\omega)$-independent case the coefficients of $A$ change after a coordinate transformation, and more importantly
the ellipticity conditions changes (unless all matrices $M_{jk}$ are hermitian, in which case the transformation is orthogonal).
On the other hand, if one does not assume the $B_j$'s generate commuting groups, then the above fails, and one needs to consider the case of general matrices. In general this leads to a $p$-dependent stochastic parabolicity condition. See \cite{DuLiuZhang,KimLeesystems} for results in this direction.
\end{remark}

\section{Divergence form equations of second order with measurable coefficients}
\label{sec:tent}
In this section, we consider the problem:
\begin{equation*}
\left\{
  \begin{array}{ll}
  dU(t) +L(t) U(t)\ud t & = f(t)\ud t +  g(t) \ud W_{H}(t), \quad \forall t>0\\
    U(0) & = 0,
  \end{array}
\right.
\end{equation*}
where $L(t) = - div\, a(t,.) \nabla$, and
$a \in L^{\infty}(\Omega \times\R_{+}; L^{\infty}(\R^{d}; \mathcal{L}(\C^{d})))$ is progressively measurable and satisfies the uniform ellipticity condition:
$$
Re \langle a(\omega, t,x)\xi,\overline{\xi} \rangle \geq C |\xi|^{2} \quad \forall \xi \in \C^{d},
$$
for almost every $(\omega, t,x) \in \Omega \times \R_{+} \times \R^{d}$, and $f,g$ belong to appropriate tent spaces defined below.\\
The idea of using tent spaces as solution spaces for stochastic PDE goes back to \cite{anp}. It can be seen as part of the trend to use harmonic analysis ``beyond Calder\'on-Zygmund theory" in PDE problems with rough coefficients (see e.g. \cite{hkmp,aa} and the references therein). The results given here demonstrate how tent spaces can be used to treat problems with $L^{\infty}(\R_{+}\times \R^{d})$ coefficients, extending the time-independent results of \cite{anp}. They include a deterministic result (under no extra assumption on the coefficients) that can be seen as the first (to the best of our knowledge) extension of Lions's maximal regularity result from \cite{L} to a non Hilbertian setting. For the stochastic problem, we impose that
$$
\|L(t)u\|_{L^{2}(\R^{d})} \lesssim \|u\|_{W^{2,2}(\R^{d})} \quad
\forall u \in W^{2,2}(\R^{d}) \quad \forall t\geq 0.
$$
This holds when $a$ is divergence free, i.e. when $\sum \limits _{i=1} ^{d} a_{i,j}\partial_{i} = 0$ (almost surely) in the sense of distributions for all $j=1,..,d$. See Remark \ref{rk:divfree} for a discussion of this condition.
\\

We plan to develop the theory presented in this section in future work. We thus only include here the simplest situation that showcases how the method of proof used in this paper (particularly in Theorem \ref{thm:timedepSMR}), as well as the idea of using the time weights $w_{\alpha}$ to vary regularity, can be combined with the tent space approach of \cite{anp}.
For this reason, we choose to take zero initial data (although we could add data in appropriate fractional domains, see Remark \ref{rk:init} below), and keep the equation linear (although semilinearities could be treated through fixed point arguments).\\

Our starting point is the $L^2$ theory of J.L. Lions (see \cite{Par2}, \cite{KR79}, \cite{LiuRock}), where solution spaces are the energy spaces $L^{2}(\R_{+};W^{1,2}(\R^{d}))$, forcing terms are taken in $L^{2}(\R_{+};W^{-1,2}(\R^{d}))$, and data are in $L^{2}(\R^{d})$. Lions's theory includes the existence of an evolution family $\{\Gamma(t,s) \;;\; t>s\}$ (see e.g. \cite[Chap.\,XVIII]{dl}).
In the deterministic setting, Lions's theory has been extended in \cite{AMP15} to allow data in $L^{p}(\R^{d})$.
The appropriate solution spaces then turn out to be tent spaces $T^{p,2} _{0}$, in the sense that there is a norm equivalence
$$
\|u_{0}\|_{L^{p}(\R^{d})} \eqsim \|(t,x)\mapsto \sqrt{t}\nabla\Gamma(t,0)u_{0}(x)\|_{T^{p,2}_{0}};
$$
see \cite[Corollary 7.5]{AMP15}. Let us recall the definition of these tent spaces, and more generally of their Sobolev counterparts.

\begin{definition}
Let $p \in [1,\infty)$ and $\sigma\geq 0$. Let $K$ be a Hilbert space.
The tent space $T^{p,2} _{\sigma}(K)$ is defined as the completion of $C^{\infty} _{c} (\R_{+} \times \R^{d};K)$ with respect to the norm
$$
\|g\|_{T_{\sigma }^{p,2}(K)} := \Big( \int  _{\R^{d}}
\Big(\int  _{0} ^{\infty} \fint  _{B(x,t^\frac12)}\|g(y,t)\|_{K}^{2} \,dy\,
\frac{dt}{t^{1+\sigma}} \Big)^{\frac{p}{2}} \,dx \Big)^{\frac{1}{p}}.
$$
When $K=\C$, we just write $T^{p,2} _{\sigma}$ instead of $T^{p,2} _{\sigma}(\C)$.
\end{definition}
Note that, by Fubini's theorem, $T^{2,2,\sigma} = L^{2}(\R_{+}\times \R^{d}, \frac{dt dx}{t^{1+\sigma}})$.

As proven in \cite{a-angle}, the aperture can be changed in the definition of $T_{\sigma }^{p,2}$ in the following way. There exists $C>0$ such that for $\alpha\geq 1$ and all $g \in T^{p,2} _{\sigma}(K)$:
\begin{equation}
\label{eq:change}
\Big( \int  _{\R^{n}}
\Big(\int  _{0} ^{\infty} \fint  _{B(x,\alpha t^\frac12)}\|g(y,t)\|_{K}^{2} \,dy\,
\frac{dt}{t^{1+\sigma}} \Big)^{\frac{p}{2}} \,dx \Big)^{\frac{1}{p}}
\leq C \alpha^{\frac{d}{\min(p,2)}} \|g\|_{T^{p,2}_{\sigma}(K)}.
\end{equation}
We think of $T^{p,2}_{\sigma}$ has being to $T^{p,2}_{0}$ what $W^{\sigma,p}$ is to $L^p$. This is motivated by the following classical fact from Littlewood-Paley theory: for all $p>1$ and all $u_{0} \in L^{p}(\R^{d})$, we have that
$$
\|u_{0}\|_{L^{p}(\R^{d})} \eqsim \|(t,x)\mapsto (-t\Delta)^{\frac{1}{2}} \exp(t\Delta)u_{0}(x)\|_{T^{p,2} _{0}}.
$$
This follows from the Hardy space theory introduced in \cite{fs} and interpolation.
Therefore, for $\sigma \geq 0$ and $f \in H^{\sigma,p}$, we have that
\begin{align*}
\|(-\Delta)^{\frac{\sigma}{2}} f\|_{L^{p}(\R^{d})}
&\eqsim \|(t,x) \mapsto (-t\Delta)^{\frac{1}{2}} \exp(t\Delta)(-\Delta)^{\sigma}f(x)\|_{T^{p,2}_{0}}\\
 &= \|(t,x) \mapsto (-t\Delta)^{\frac{1+\sigma}{2}} \exp(t\Delta)f(x)\|_{T^{p,2}_{\sigma}}
\\ & \eqsim \|(t,x) \mapsto (-t\Delta)^{\frac{1}{2}} \exp(t\Delta)f(x)\|_{T^{p,2}_{\sigma}},
\end{align*}
where the last equivalence comes from a change of square function argument involving Schur's Lemma (see e.g. \cite[Theorem 7.10]{hnp}). Note that, in the time dependent setting, \cite[Corollary 7.5]{AMP15} gives that there exists $p_{c}<2$ such that for all $p>p_{c}$,
$$
\|f\|_{L^{p}(\R^{d})} \eqsim \|(t,x)\mapsto \sqrt{t}\nabla \Gamma(t,0)f(x)\|_{T^{p,2}_{0}}.
$$
At this stage, we do not know if, more generally for $\sigma \geq 0$,
$$\|(-\Delta)^{\frac{\sigma}{2}} f\|_{L^{p}(\R^{d})}  \eqsim \|(t,x)\mapsto \sqrt{t}\nabla \Gamma(t,0)f(x)\|_{T^{p,2}_{\sigma}}.$$

In this section we thus use well chosen combinations of powers of $div \, a \nabla$ and the parameter $\sigma$ to measure regularity.
This is in the spirit of Amenta-Auscher's tent space approach to elliptic boundary value problems with fractional regularity \cite{aa} (where the link between $\sigma$ and powers of the relevant Dirac operator is completely understood).

Our main result is the following theorem, proven at the end of the section.
\begin{theorem}
\label{thm:tent}
Let  $\sigma \geq 0$ and $p>\min(1,\frac{2d}{d+2\sigma+2})$.
Let $f \in L^{p}(\Omega;T^{p,2}_{\sigma})$, and $g \in L^{p}(\Omega;T^{p,2}_{\sigma+1}(H))$ be an adapted process such that
$\nabla g \in L^{p}(\Omega;T^{p,2}_{\sigma}(H^{d}))$.
Under Assumption \eqref{ass}, we have that the solution process defined by
$$U(t,.) = \int \limits _{0} ^{t}
\Gamma(t,s)f(s)ds + \int \limits _{0} ^{t}
(\Gamma(t,s)\otimes I_{H})g(s)dW_{H}(s) \quad \forall t>0,$$
satisfies
$$
\mathbb{E}\|U\|_{T^{p,2}_{\sigma+2}}^{p}
\lesssim \E\|f\|_{T^{p,2}_{\sigma}} ^{p}
+\E\|g\|_{T^{p,2}_{\sigma+1}(H)} ^{p} + \E\|\nabla g\|_{T^{p,2}_{\sigma}(H^{d})} ^{p}.
$$
\end{theorem}

\begin{remark}
\label{rk:maximal?}
(1) To compare such tent space estimates to $L^{p}(L^q)$ regularity, one should note that, by \cite[Proposition 2.1]{ahm}, we have that, for all $\sigma \geq 0$ and all $p \geq 2$, there exists $C>0$ such that for all $F \in L^{p}(\R^{d};L^{2}(\R_{+},\frac{dt}{t^{1+\sigma}}))$:
$$
\|F\|_{T^{p,2}_{\sigma}} \leq C \|F\|_{L^{p}(\R^{d};L^{2}(\R_{+},\frac{dt}{t^{1+\sigma}}))}.
$$
The reverse inequality holds for $p \leq 2$.
This means that for $p \geq 2$ the tent space theory allows more forcing terms than its $L^{p}(L^2)$ analogue. For $p<2$, the class of forcing terms $T^{p,2}$ is smaller than its $L^{p}(L^{2})$ counterpart, which it has to be, given that maximal regularity in $L^{p}(L^{2})$ does not hold for $p<2$ (see \cite{Kry}).\\
(2) On the other hand, we estimate $\|U\|_{T^{p,2}_{\sigma+2}}$ in terms of $\|\nabla g\|_{T^{p,2}_{\sigma}(H^{d})}$. A more natural generalisation of Lions's maximal regularity result would be to estimate $\|\nabla U\|_{T^{p,2}_{\sigma}(H^{d})}$ in terms of $\|g\|_{T^{p,2}_{\sigma}}$ or $\|U\|_{T^{p,2}_{\sigma}}$ in terms of  $\|G\|_{T^{p,2}_{\sigma}(H^{d})}$ for $g =div G$. At this stage, we would need to know that $\{(t-s)^{\frac{1}{2}}\nabla \Gamma(t,s) \;;\; t>s\}$ is uniformly bounded in $\mathcal{L}(L^{2}(\R^{d}),L^{2}(\R^{d};\C^{d}))$ to study such maximal regularity. The fact that we only know uniform boundedness of $\{\Gamma(t,s) \;;\; t>s\}$ is what led us to the notion of maximal regularity used here. Note, however, that this uniform boundedness of the evolution family is, in Lions's theory, a consequence of the energy estimates that also give the $L^{2}(W^{-1,2})-L^{2}(W^{1,2})$ maximal regularity. So, while
$L^{2}(\frac{dt}{t^{\sigma}};L^2)-L^{2}(\frac{dt}{t^{2+\sigma}};L^2)$ maximal regularity is trivial in the time independent case, it is not so in the time dependent case, where generation of a bounded evolution family is not a substantially easier question than maximal regularity. Nevertheless, we plan to return to the question of estimating $\|\nabla U\|_{T^{p,2}_{\sigma}(H^{d})}$ in future work.
\end{remark}

To prove Theorem \ref{thm:tent}, we proceed as in Theorem \ref{thm:timedepSMR}, and decompose the problem into a time-independent stochastic part, and a time-dependent deterministic part. Our key technical tool
to estimate both parts is extrapolation in tent spaces, as developed in \cite{akmp} and \cite{anp}. In particular, we need some simple variations of \cite[Proposition 5.1]{anp}, proven below (we include the details for the convenience of the reader).
These results make extensive use of the notion of $L^2 - L^2$ off-diagonal decay.

\begin{definition}
A family of bounded linear operators $\{K(t,s) \;;\; t>s\} \subset \mathcal{L}(L^{2}(\R^{d}))$ is said to have $L^2 - L^2$ off-diagonal decay of any order if, for each $m \in \N$, there exists $C_{m}>0$ such that, for every Borel sets $E,F \subset \R^{d}$, every $u \in L^{2}(\R^{d})$, every $t>s$, we have that
$$
\|1_{E} K(t,s)(1_{F}u)\|_{2} \leq C_{m}\big(1+\frac{d(E,F)^{2}}{t-s}\big)^{-m} \|1_{F}u\|_{2},
$$
where $d(E,F) = \inf\{d(x,y) \;;\; x \in E, \, y \in F\}$.
A one parameter family $\{K(t) \;;\; t>0\} \subset \mathcal{L}(L^{2}(\R^{d}))$ is said to have $L^2 - L^2$ off-diagonal decay of any order if the two parameter family $\{K(t-s) \;;\; t>s\}$ does.
\end{definition}

\begin{proposition}
\label{prop:det-extra}
Let $\{K(t,s) \;;\; t>s\} \subset \mathcal{L}(L^{2}(\R^{d}))$ be uniformly bounded with $L^2 - L^2$ off-diagonal decay of any order. Assume that
$$\mathcal{M}_{K} f(t,.) = \int \limits _{0} ^{t} K(t,s)f(s)ds,$$
defines a bounded linear operator from $T^{2,2}_{\sigma}$
to $T^{2,2}_{\sigma +2}$ for all $\sigma \geq 0$. Then, for all $\sigma \geq 0$ and all $p>\min(1,\frac{2d}{d+2\sigma+4})$, $\mathcal{M}_{K}$ extends to a bounded linear operator from $T^{p,2}_{\sigma}$
to $T^{p,2}_{\sigma +2}$.
\end{proposition}

\begin{proof}
We introduce the sets
$$
C_j(x,t) = \left\{
\begin{array}{ll}
  B(x,t) &  \ j=0 \\
   B(x,2^jt)\setminus B(x,2^{j-1}t) &  \ j=1,2,\dots
\end{array}
\right.
$$
Let $t>0$ and $f \in T^{2,2} _{\sigma} \cap T^{p,2} _{\sigma}$.
\begin{equation*}
\begin{aligned}
\ & \Big\| (t,x)\mapsto \int  _{0} ^t K(t,s)f(s,\cdot)(x)\,ds\Big\|_{T^{p,2}_{\sigma+2}} \leq \sum_{j=0}^\infty \sum_{k=1}^\infty I_{j,k} +  \sum_{j=0}^\infty J_j,
\end{aligned}
\end{equation*}
where
$$ I_{j,k} ^{p} = 2^{-\frac{kp}{2}}\int_{\R^n} \!\Big(
 \int_0^\infty\!\! \fint_{B(x,t^\frac12)} \int^{2^{-k}t}_{2^{-k-1}t}
| K(t,s) [\one_{C_j(x,4t^\frac12)} f(s,\cdot)](y)|^2\,\,ds\,dy\frac{dt}{t^{\sigma+2}}
\Big)^\frac{p}{2}dx
$$
and
$$ J_j ^{p} = \int_{\R^n} \!\Big(
 \int_0^\infty\!\! \fint_{B(x,t^\frac12)} \Big(\int^t_{\frac{t}{2}}
K(t,s) [\one_{C_j(x,4s^\frac12)} f(s,\cdot)](y)ds \Big)^2\,\,dy\frac{dt}{t^{\sigma+3}}
\Big)^\frac{p}{2}dx.
$$
Let us start with the estimate for $I_{j,k}$ for $j\ge 0$ and $k\ge 1$. Using the off-diagonal decay, we have the following, for all $x\in \R^{d}$:
\begin{align*}
\ &  \int_0^\infty \!\!\!\fint_{B(x,t^\frac12)} \int^{2^{-k}t}_{2^{-k-1}t}
|K(t,s) [\one_{C_j(x,4t^\frac12)} f(s,\cdot)](y)|^2\,\,ds\,dy\frac{dt}{t^{\sigma+2}}
\\ & \lesssim  \int_0^\infty\!\!\! \int^{2^{-k}t}_{2^{-k-1}t}
\big\n \one_{B(x,t^\frac12)} K(t,s)
[\one_{C_j(x,4t^\frac12)} f(s,\cdot)]\big\n_{L^2}^2\,\,ds\frac{dt}{t^{\frac{d}{2}+\sigma+2}}
\\ & \lesssim   \int_0^\infty\!\!\! \int^{2^{-k}t}_{2^{-k-1}t}
(\frac{4^{j}t}{t-s})^{-d}\big\n
\one_{C_j(x,4t^\frac12)} f(s,\cdot)\big\n_{L^2}^2\,\,ds\frac{dt}{t^{\frac{d}{2}+\sigma+2}}
\\ & \lesssim 4^{-jd} \int_0^\infty\Big( \int^{2^{k+1}s}_{2^k s} \frac{dt}{t^{\frac{d}{2}+\sigma+2}}\Big)
\big\n \one_{B(x,2^{j+\frac{k}{2}+3}s^\frac12)} f(s,\cdot)\big\n_{L^2}^2\,ds
\\ & \lesssim 4^{-jd}2^{-k(\frac{d}{2}+\sigma+1)} \int_0^\infty
\big\n \one_{B(x,2^{j+\frac{k}{2}+3}s^\frac12)} f(s,\cdot)\big\n_{L^2}^2\,\frac{ds}{s^{\frac{d}{2}+\sigma+1}}.
\end{align*}
We then use the change of aperture formula in tent spaces \eqref{eq:change} to get that
\begin{align*}
I_{j,k}
\lesssim 2^{-jd}2^{-\frac12k(\frac{d}{2}+\sigma+2)}2^{(j+\frac{k}{2}+3)\frac{d}{p\wedge 2}}
\|f\|_{T^{p,2}_{\sigma}}.
\end{align*}
The sum  $\sum_{j,k} I_{j,k}$ thus converges since we assumed that $p> \frac{2d}{d+2\sigma+4}$.

Turning to $J_0$, and using the fact the $M$ is bounded from $T^{2,2}_{\sigma+\frac{d}{2}}$ to $T^{2,2}_{\sigma +2+\frac{d}{2}}$, we have that
\begin{align*}
\ &  \int_0^\infty \!\!\!\fint_{B(x,t^{\frac{1}{2}})} |\int_{\frac{t}{2}}^{t} K(t,s) [\one_{B(x,4s^\frac12)} f(s,\cdot)](y)ds|^{2}\,dy\,\frac{dt}{t^{\sigma+3}}
\\ & \lesssim \int_0^\infty \!\!\!\fint_{B(x,t^{\frac{1}{2}})} |\int_{0}^{t} K(t,s) [\one_{B(x,4s^\frac12)} f(s,\cdot)](y)ds|^{2}\,dy\,\frac{dt}{t^{\sigma+3}}
+ \sum \limits _{k=0} ^{\infty} I_{0,k}
\\ & \lesssim \|\mathcal{M}_{K} \left((s,y) \mapsto \one_{B(x,4s^\frac12)}(y) f(s,y) \right)\|_{T^{2,2}_{\sigma+2+\frac{d}{2}}}+ \sum \limits _{k=0} ^{\infty} I_{0,k}
\\ & \lesssim \int_0^\infty \!\!\!\fint_{B(x,s^{\frac{1}{2}})}
|f(s,y)|^{2} dy\,\frac{ds}{s^{\sigma+1}}+ \sum \limits _{k=0} ^{\infty} I_{0,k}.
\end{align*}
It remains to estimate $J_j$ for $j\ge 1$. We have
\begin{align*}
\ &  \int_0^\infty \!\!\! \fint_{B(x,t^\frac12)} |\int^{t}_{\frac{t}{2}} K(t,s) [\one_{C_j(x,4s^\frac12)} f(s,\cdot)](y)ds|^{2}\,dy\frac{dt}{t^{\sigma+3}}
\\& \lesssim \int_0^\infty \!\!\! \fint_{B(x,t^\frac12)} \int^{t}_{\frac{t}{2}} |K(t,s) [\one_{C_j(x,4s^\frac12)} f(s,\cdot)](y)|^{2} \,ds\,dy\frac{dt}{t^{\sigma+2}}
\\ & \lesssim
\int_0^\infty\!\!\!  \int^{t}_{\frac{t}{2}} (\frac{4^{j}s}{t-s})^{-d}
\n \one_{B(x,2^{j+2}s^\frac12)} f(s,\cdot)\n_{L^2}^2\,\,ds\,\frac{dt}{t^{\frac{d}{2}+\sigma+2}}
\\ & \le 4^{-jd}
\int_0^\infty \!\!\int^{t}_{\frac{t}{2}}
\n \one_{B(x,2^{j+2}s^\frac12)} f(s,\cdot)\n_{L^2}^2\,\frac{ds}{s^{\frac{d}{2}+\sigma+2}}\,dt
\\ & =4^{-jd}
\int_0^\infty \!\!\Big(\int^{2s}_{s}\,dt\Big)
\n \one_{B(x,2^{j+2}s^\frac12)} f(s,\cdot)\n_{L^2}^2\,\frac{ds}{s^{\frac{d}{2}+\sigma+2}}
\\ & \lesssim
4^{-jd} \int_0^\infty \!\!
\n \one_{B(x,2^{j+2}s^\frac12)} f(s,\cdot)\n_{L^2}^2\,\frac{ds}{s^{\frac{d}{2}+\sigma+1}}.
\end{align*}

Using the change of aperture formula \eqref{eq:change} one more time, we have that
\begin{align*}
J_{j}
& \lesssim 2^{-jd} 2^{(j+2)\frac{d}{p\wedge 2}}\n f\n_{T_{\sigma}^{p,2}},
\end{align*}
which concludes the proof.
\end{proof}

In a very similar way, we have the following stochastic version:

\begin{proposition}
\label{prop:stoch-extra}
Let $\{K(t,s) \;;\; t>s\}$ be a two-parameter progressively measurable process with values in $\mathcal{L}(L^{2}(\R^{d}))$. Assume that, almost surely, $\{K(t,s) \;;\; t>s\}$ is uniformly bounded with $L^2 - L^2$ off-diagonal decay of any order. Assume further that
$$\mathcal{S}_{K} g(t,.) = \int \limits _{0} ^{t} (K(t,s)\otimes I_{H})g(s)dW_{H}(s),$$
defines a bounded linear operator from $L^{2}(\Omega;T^{2,2}_{\sigma+1}(H))$
to $L^{2}(\Omega;T^{2,2}_{\sigma+2})$ for all $\sigma \geq 0$. Then, for all $\sigma \geq 0$ and $p>\min(1,\frac{2d}{d+2\sigma+2})$, $\mathcal{S}_{K}$ extends to a bounded linear operator from $L^{p}(\Omega;T^{p,2}_{\sigma+1}(H))$
to $L^{p}(\Omega;T^{p,2}_{\sigma+2})$.
\end{proposition}

\begin{proof}
Reasoning as in \cite[Proposition 5.1]{anp}, and using
the It\^o isomorphism for stochastic integrals
(Theorem \ref{thm:UMD}) as well as a
square function estimate \cite[Proposition 6.1]{NVWco}, we obtain
\begin{equation*}
\begin{aligned}
\ & \mathbb{E}\Big\|(t,x) \mapsto \int  _{0} ^t (K(t,s)\otimes I_{H})g(s,\cdot)(x)\,dW_{H}(s)\Big\|_{T^{p,2} _{\sigma+2}}^{p}
\\ & \quad \eqsim
\E \int  _{\R^{d}}
\Big(\int  _{0} ^{\infty} \fint_{B(x,t^\frac12)}
\int_0^t
\|(K(t,s)\otimes I_{H})[g(s,\cdot)](y)\|_{H} ^{2}\,ds \,dy\,\frac{dt}{t^{\sigma+3}} \Big)^{\frac{p}{2}} \,dx.
\end{aligned}
\end{equation*}
One can then copy the proof of Proposition \ref{prop:det-extra} with $\sigma$ replaced by $\sigma +1$, undoing the It\^o isometry when handling the $J_{0}$ term.
The only difference is that there is no factor $2^{-\frac{kp}{2}}$ in front of the $I_{j,k}$ term (as we use It\^o's isomorphism instead of Cauchy-Schwarz inequality). This is the reason why one needs $p>\min(1,\frac{2d}{d+2\sigma+2})$ rather than
$p>\min(1,\frac{2d}{d+2\sigma+4})$.
\end{proof}

\subsection{Deterministic time-dependent maximal regularity in $T^{p,2}_{\sigma}$.}

We consider the problem
\begin{equation}\label{eq:divform-det}
\left\{
  \begin{array}{ll}
  \partial_{t} u(t,x) - div\, a(t,.) \nabla u(t,x) & = f(t,x), \quad t \geq 0,x \in \R^{d},\\
    u(0) & = 0,
  \end{array}
\right.
\end{equation}
for $f \in T^{2,2}_{s} \cap T^{p,2}_{s}$.
Our goal is to derive a priori maximal regularity estimates for the solution given by
$$
u(t,.) = \int \limits _{0} ^{t} \Gamma(t,s)f(s,.)ds,
$$
where $\{\Gamma(t,s) \;;\; t>s\} \subset \mathcal{L}(L^{2}(\R^{d}))$ is Lions's evolution family.\\

To do so, we consider the maximal regularity operator
defined, for $f \in T^{2,2}_{-1} = L^{2}(\R_{+}\times \R^{d})$ by
$$
\mathcal{M}f(t,.). = \int \limits _{0} ^{t} \Gamma(t,s)f(s,.)ds
$$
Thanks to the uniform boundedness of $\Gamma(t,s) \in \mathcal{L}(L^{2}(\R^{d}))$ (see e.g. \cite[Lemma 3.14]{AMP15}), we have the following mapping properties in the $T^{2,2}_{\sigma}$ scale. See Remark \ref{rk:maximal?} for a discussion of the meaning of this notion of maximal regularity.

\begin{proposition}
\label{prop:HardyL2}
The operator $\mathcal{M}$, initially defined on $L^{2}(\R_{+}\times \R^{d})$, extends to a bounded linear operator from $T^{2,2}_{\sigma}$ to $T^{2,2}_{\sigma+2}$ for all $\sigma\geq 0$.
\end{proposition}

\begin{proof}
Let $f \in T^{2,2}_{\sigma}$. We have that $\mathcal{M}f=\mathcal{M}_{1}(f)+\mathcal{M}_{2}(f)$, where, for all $t>0$,
\begin{align*}
\mathcal{M}_{1}(f)(t,.) &= \int \limits _{0} ^{\frac{t}{2}} \Gamma(t,s)f(s)ds,\qquad
\mathcal{M}_{2}(f)(t,.) = \int \limits _{\frac{t}{2}} ^{t}\Gamma(t,s)f(s)ds.
\end{align*}
For $\mathcal{M}_{1}(f)$, we have the following
\begin{align*}
\|\mathcal{M}_{1}(f)\|_{T^{2,2}_{\sigma+2}} ^{2}
&= \int \limits _{0} ^{\infty} \|\int \limits _{0} ^{\frac{t}{2}}
\Gamma(t,s)f(s)ds\|_{2} ^{2} \frac{dt}{t^{\sigma +3}}\lesssim \int \limits _{0} ^{\infty} \int \limits _{0} ^{\frac{t}{2}}
\|\Gamma(t,s)f(s)\|^{2}_2 ds \frac{dt}{t^{\sigma +2}}\\
&\lesssim  \int \limits _{0} ^{\infty} \|f(s)\|_{2} ^{2} (\int \limits_{2s}  ^{\infty} \frac{dt}{t^{\sigma+2}})ds
\lesssim \int \limits _{0} ^{\infty} \|f(s)\|_{2} ^{2} \frac{ds}{s^{\sigma+1}}.
\end{align*}
To estimate $\mathcal{M}_{2}(f)$, we use Hardy inequality and the uniform boundedness of $\Gamma(t,s)$ as follows:
\begin{align*}
\|\mathcal{M}_{2}(f)\|_{T^{2,2}_{\sigma+2}} ^{2}
&= \int \limits _{0} ^{\infty} \|\int \limits _{\frac{t}{2}} ^{t}
\Gamma(t,s)f(s)ds\|_{2} ^{2} \frac{dt}{t^{\sigma +3}}
\lesssim
\int \limits _{0} ^{\infty} \big( \frac{1}{t}\int \limits _{\frac{t}{2}} ^{t}
\|\Gamma(t,s)(f(s)s^{-\frac{\sigma+1}{2}})\|_{2} ds\big) ^{2} dt\\
& \lesssim
\int \limits _{0} ^{\infty} \big( \frac{1}{t}\int \limits _{\frac{t}{2}} ^{t}
\|f(s)s^{-\frac{\sigma+1}{2}}\|_{2}ds \big) ^{2} dt
\lesssim \int \limits _{0} ^{\infty} \|f(s)\|_{2} ^{2} \frac{ds}{s^{\sigma+1}}.
\end{align*}
\end{proof}

\begin{theorem}
\label{thm:nonaut-det-tent}
Let $\sigma \geq 0$ and $p>\min(1,\frac{2d}{d+2\sigma+4})$
The deterministic maximal regularity operator $\mathcal{M}$ extends to a bounded linear operator from $T^{p,2}_{\sigma}$ to $T^{p,2}_{\sigma+2}$.
\end{theorem}

\begin{proof}
The family $\{\Gamma(t,s) \;;\; t>s\}$ has $L^2 - L^2$ off-diagonal decay  of any order by \cite[Proposition 3.19]{AMP15}. The result thus follows from Proposition \ref{prop:HardyL2} and  Proposition \ref{prop:det-extra}.
\end{proof}

\begin{remark}
As in \cite{akmp,anp,AMP15}, we could also exploit $L^{p}-L^{2}$ off-diagonal decay for $p<2$ (and even $p=1$ if the coefficients are real valued). This would give a wider range of $p$ (and the full range $(1,\infty)$ if the coefficients are real valued). We leave this technical improvement for future work. \end{remark}

\subsection{Stochastic time-independent maximal regularity in $T^{p,2}_{\sigma}$.}
In this subsection we consider the problem
\begin{equation*}
\left\{
  \begin{array}{ll}
  dU(t) - \Delta U(t)\ud t & = g(t) \ud W_{H}(t), \\
    U(0) & = 0,
  \end{array}
\right.
\end{equation*}
where $\sigma \geq 0$, and $g$ is an adapted process in $L^{p}(\Omega;T^{p,2} _{\sigma}(H))$ such that
$\nabla g \in L^{p}(\Omega;T^{p,2} _{\sigma}(H^{d}))$.
Propositions 4.1 and 5.1(applied with the kernel
$K(t,s) = \exp((t-s)\Delta)div$)) from \cite{anp} give that the mild solution defined by
$$U(t) = \int \limits _{0} ^{t} \exp((t-s)\Delta)g(s) \ud W_{H}(s)$$
satisfies
\begin{equation}
\label{eq:anp}
\E\|\Delta U\|_{T^{p,2} _{\sigma}} ^{p} \lesssim \E\|\nabla g\|_{T^{p,2} _{\sigma}(H^{d})} ^{p},
\end{equation}
\begin{remark}
\label{rk:init}
As in \cite[Lemma 6.3]{anp}, we could add initial data $u_{0} \in L^{\frac{\beta}{2}}$ for appropriate values of $\beta$. To do so, one needs to modify the proof of \cite[Lemma 6.3]{anp} to control $L^{\frac{1}{2}}\exp(-tL)u_{0}$ instead of $\nabla \exp(-tL)u_{0}$.
\end{remark}

\subsection{Stochastic time-dependent maximal regularity in $T^{p,2}_{\sigma}$.}
We now combine the previous results to treat our main problem
\begin{equation*}
\left\{
  \begin{array}{ll}
  dU(t) +L(t) U(t)\ud t & = f(t)\ud t +  g(t) \ud W_{H}(t), \\
    U(0) & = 0,
  \end{array}
\right.
\end{equation*}
We choose not to include initial data, but could do so as indicated in the above remark.
To allow the approach used in Theorem \ref{thm:timedepSMR} to
work here, we need to make the following assumption on our $L(t,\omega)=-div a(t,\omega,.) \nabla$ operators:
\begin{assumption}
\label{ass}
There exists $C>0$ such that for all $t \geq 0$, all $\omega \in \Omega$, and all $u \in W^{2,2}(\R^{d})$,
$$
\|L(t,\omega)u\|_{L^{2}(\R^{d})} \leq C \|u\|_{W^{2,2}(\R^{d})}.$$
\end{assumption}

\begin{remark}
\label{rk:divfree}
Assumption \ref{ass} can be satisfied by coefficients that do not have any regularity in space or time. Indeed, it holds for all divergence free coefficients, i.e. coefficients such that $\sum \limits _{i=1} ^{d} a_{i,j}\partial_{i} = 0$ (almost surely) in the sense of distributions for all $j=1,..,d$.
This was first remarked (to the best of our knowledge) in \cite[Lemma 4.4]{ers}. Example of divergence free coefficients include, for $d=3$, matrices for which columns are of the form $curl F$ for some Lipschitz vector field $F$. Since Assumption \ref{ass} is also satisfied when the coefficients $a$ are Lipschitz continuous in space (by the product rule and Riesz transform boundedness), we have that Assumption \ref{ass} holds for all coefficients of the form $b+c$ where $b \in L^{\infty}(\Omega \times \R_{+};W^{1,\infty}(\R^{d}))$ and $c \in L^{\infty}(\Omega\times \R_{+} \times \R^{d})$ is divergence free.
\end{remark}

\begin{lemma}
\label{lem:od}
Under Assumption \eqref{ass}, we have that
$\{tL(t,\omega)(I-t\Delta)^{-1} \;,\; t>0\}$ has $L^{2}-L^{2}$ off-diagonal decay of any order, uniformly in $\omega \in \Omega$.
\end{lemma}

\begin{proof}
Let $E,F \subset \R^{d}$ be two Borel sets such that $d(E,F)>0$. Let $\eta \in C^{\infty}(\R^{d})$ be such that
$\eta(x)=1$ for all $x \in E$, $\eta(x) = 0$ for all $x \not \in \tilde{E}=\{y \in \R^{d} \;;\; d(y,E) \leq \frac{d(E,F)}{2}\}$, and $\|\nabla \eta\|_{\infty} \leq \frac{1}{d(E,F)}$.
Note that Assumption \ref{ass} implies that
$$
\underset{t\in \R_{+},\omega \in \Omega}{\sup}
\|tL(t,\omega)(I-t\Delta)^{-1}\|_{\mathcal{L}(L^{2}(\R^{d})}<\infty.
$$

For $u \in L^{2}(\R^{d})$, $t>0$, and $\omega \in \Omega$, we thus have that
\begin{align*}
\|1_{E} t & L(t,\omega) (I-t\Delta)^{-1}(1_{F}u)\|_{2}
= \|1_{E} t L(t,\omega)\eta (I-t\Delta)^{-1}(1_{F}u)\|_{2}
\\
& \leq \|tL(t,\omega) (I-t\Delta)^{-1}(I-t\Delta)(\eta(I-t\Delta)^{-1}(1_{F}u))\|_{2} \\
& \lesssim \|(I-t\Delta)(\eta(I-t\Delta)^{-1}(1_{F}u))\|_{2} \\
& = \|1_{\tilde{E}} (I-t\Delta)(\eta(I-t\Delta)^{-1}(1_{F}u))\|_{2}
\\
& \leq \|1_{\tilde{E}} (I-t\Delta)^{-1}(1_{F}u)\|_{2}
+\|1_{\tilde{E}} t\Delta(\eta(I-t\Delta)^{-1}(1_{F}u))\|_{2}
\end{align*}
Since $\{(I-t\Delta)^{-1} \;;\; t>0\}$ has $L^{2}-L^{2}$ off-diagonal decay of any order (see e.g. \cite{AKMc} or just use standard heat kernel bounds), we only need to consider the second term. From Leibnitz rule, we have that, for all $v \in D(L)$,
$$
\Delta (\eta v) =
\eta \Delta + 2\nabla \eta . \nabla v + div (v\nabla \eta).
$$
Therefore
\begin{align*}
\|1_{\tilde{E}} & t\Delta(\eta(I-t\Delta)^{-1}(1_{F}u))\|_{2} \\
& \lesssim
\|1_{\tilde{E}} t\Delta(I-t\Delta)^{-1}(1_{F}u)\|_{2}
+ \|\nabla \eta\|_{\infty} \|t\nabla (I-t\Delta)^{-1}(1_{F}u)\|_{2} \\
& \qquad \qquad
+ \|1_{\tilde{E}} tdiv ((\nabla {\eta})(I-t\Delta)^{-1}(1_{F}u))\|_{2}
\\ & \lesssim
\|1_{\tilde{E}} t\Delta(I-t\Delta)^{-1}(1_{F}u)\|_{2}
+ \frac{\sqrt{t}}{d(E,F)} \|\sqrt{t}\nabla (I-t\Delta)^{-1}(1_{F}u)\|_{2}\\
&\qquad \qquad
+ \|1_{\tilde{E}} tdiv((\nabla \eta)(I-t\Delta)^{-1}(1_{F}u))\|_{2}
\end{align*}
Since $\{t\Delta(I-t\Delta)^{-1} \;;\; t>0\}$ and $\sqrt{t}\nabla (I-t\Delta)^{-1}$ have $L^{2}-L^{2}$ off-diagonal decay of any order by \cite[Proposition 5.2]{AKMc} (or just heat kernel estimates), and since one can assume that $\frac{\sqrt{t}}{d(E,F)}\leq 1$ without loss of generality (if $\frac{\sqrt{t}}{d(E,F)}> 1$, the estimates follow directly from Assumption \ref{ass}), we only need to consider the last term. Using Leibnitz rule again, we have that
\begin{align*}
\|1_{\tilde{E}} & tdiv (\nabla {\eta}(I-t\Delta)^{-1}(1_{F}u))\|_{2}
\\ & \leq
\underset{i,j=1,...,d}{\max} \big( \|t(\partial_{i} \partial_{j}\eta)\|_{\infty} \|1_{\tilde{E}}(I-t\Delta)^{-1}(1_{F}u)\|_{2} \\
& \qquad \qquad
+ \|\sqrt{t}\partial_{j}\eta\|_{\infty} \|1_{\tilde{E}}\sqrt{t}\nabla(I-t\Delta)^{-1}(1_{F}u)\|_{2}\big),
\end{align*}
which concludes the proof since $\{(I-t\Delta)^{-1} \;;\; t>0\}$ and $\sqrt{t}\nabla (I-t\Delta)^{-1}$ have $L^{2}-L^{2}$ off-diagonal decay of any order by \cite[Proposition 5.2]{AKMc} (or just standard heat kernel bounds).
\end{proof}

We can now prove Theorem \ref{thm:tent}, which statement we recall here.

\begin{theorem}
Let  $\sigma \geq 0$ and $p>\min(1,\frac{2d}{d+2\sigma+2})$.
Let $f \in L^{p}(\Omega;T^{p,2}_{\sigma})$, and $g \in L^{p}(\Omega;T^{p,2}_{\sigma+1}(H))$ be an adapted process such that
$\nabla g \in L^{p}(\Omega;T^{p,2}_{\sigma}(H^{d}))$.
Under Assumption \eqref{ass}, we have that the solution process defined by
$$U(t,.) = \int \limits _{0} ^{t}
\Gamma(t,s)f(s)ds + \int \limits _{0} ^{t}
(\Gamma(t,s)\otimes I_{H})g(s)dW_{H}(s) \quad \forall t>0,$$
satisfies
$$
\mathbb{E}\|U\|_{T^{p,2}_{\sigma+2}}^{p}
\lesssim \E\|f\|_{T^{p,2}_{\sigma}} ^{p}
+\E\|g\|_{T^{p,2}_{\sigma+1}(H)} ^{p} + \E\|\nabla g\|_{T^{p,2}_{\sigma}(H^{d})} ^{p}.
$$
\end{theorem}

\begin{proof}
As in the proof of Theorem \ref{thm:timedepSMR}, we decompose $U$ as $U= V_{1} + V_{2}$, where
\begin{align*}
V_{1}(t,.) &=  \int \limits _{0} ^{t} \exp((t-s)\Delta) g(s)dW_{H}(s) \quad \forall t>0, \\
V_{2}(t,.) & = \int \limits _{0} ^{t} \Gamma(t,s)f(s)ds +
\int \limits _{0} ^{t} (L(s)-\Delta)V_{1}(s)ds \quad \forall t>0.
\end{align*}
Applying the deterministic estimate (Theorem \ref{thm:nonaut-det-tent}) pathwise, along with the stochastic time-independent estimate from \cite{anp} \eqref{eq:anp}, we have that
\begin{align*}
\E\|V_{2}\|_{T^{p,2}_{\sigma+2}} ^{p}
&\lesssim \E\|f\|_{T^{p,2}_{\sigma}} ^{p} +
\E\|\Delta V_{1}\|_{T^{p,2}_{\sigma}} ^{p}
+\E\|L(.)V_{1}\|_{T^{p,2}_{\sigma}} ^{p}\\
& \lesssim \E\|f\|_{T^{p,2}_{\sigma}} ^{p} +
\E\|\nabla g\|_{T^{p,2}_{\sigma}(H^{d})} ^{p}
+\E\|L(.)V_{1}\|_{T^{p,2}_{\sigma}} ^{p}.\\
\end{align*}
To estimate the last term, we use Lemma \ref{lem:od} and
\cite[Theorem 5.2]{hnp} (where the case $\sigma=0$ is treated but the proof extends verbatim to $\sigma>0$) to obtain
\begin{align*}
\E\|L(.)V_{1}\|_{T^{p,2}_{\sigma}} ^{p}
&\lesssim
\E\|(t,x) \mapsto tL(t)(I-t\Delta)^{-1}(t^{-1}-\Delta)V_{1}(t,x)\|_{T^{p,2}_{\sigma}} ^{p}\\
&\lesssim \E \|(t,x) \mapsto t^{-1}V_{1}(t,x)\|_{T^{p,2}_{\sigma}} ^{p} + \E\|\Delta V_{1}\|_{T^{p,2}_{\sigma}} ^{p}.
\end{align*}
To conclude the proof, we use \eqref{eq:anp} again to evaluate the last term, and treat the remaining term through an application of Propositions \ref{prop:det-extra} and \ref{prop:stoch-extra}. This is possible because $\{\exp((t-s)\Delta) \;;\; t>s\}$ has $L^{2}-L^{2}$ off-diagonal decay of any order, and because, by It\^o isometry,
\begin{align*}
\E\|(t,x) \mapsto & t^{-1} \int \limits _{0} ^{t} \exp((t-s)\Delta)g(s,.)(x)dW_{H}(s)\|_{T^{2,2}_{\sigma}} ^{2} \\\
& \eqsim \E \int \limits _{0} ^{\infty} \int \limits _{0} ^{t}
\|\exp((t-s)\Delta)g(s,.)\|_{L^2(\R^{d};H)} ^{2} ds \frac{dt}{t^{\sigma+3}} \\
& \leq  \E \int \limits _{0} ^{\infty} (\int \limits _{s} ^{\infty} \frac{dt}{t^{\sigma+3}})
\|g(s,.)\|_{L^2(\R^{d};H)} ^{2} ds \lesssim \|g\|_{T^{2,2}_{\sigma+1}(H)} ^{2}.
\end{align*}
We thus have that
\begin{align*}
\E\|L(.)V_{1}\|_{T^{p,2}_{\sigma}} ^{p}
\lesssim
\E\|g\|_{T^{p,2}_{\sigma+1}(H)}^{p} +
\E\|\nabla g\|_{T^{p,2}_{\sigma}(H^{d})} ^{p},
\end{align*}
which concludes the proof.
\end{proof}

\bibliographystyle{alpha}

\bibliography{literature}

\end{document}